\documentclass[11pt]{amsart}
\usepackage{amsmath}
\usepackage{amssymb}
\usepackage{amsthm}
\usepackage{amscd}
\usepackage{amsxtra}
\usepackage{verbatim}
\usepackage{xcolor}
\usepackage{color}
\usepackage{enumerate}
\usepackage{mathrsfs}

\allowdisplaybreaks

\numberwithin{equation}{section}

\definecolor{dblue}{rgb}{0,0,0.45}
\definecolor{red}{rgb}{0.7,0,0}

 \RequirePackage{geometry}
\newtheorem{theorem}{Theorem}[section]
\newtheorem{lemma}[theorem]{Lemma}
\newtheorem*{lemma*}{Lemma}

\newtheorem{proposition}[theorem]{Proposition}

\theoremstyle{definition}

\newtheorem{remark}[theorem]{Remark}

\theoremstyle{remark}

\newcommand{\N}{{\mathbb N}}
\newcommand{\R}{{\mathbb R}}

\begin{document}

\title{Conditions for Boundedness into Hardy spaces}

\author[Grafakos]{Loukas Grafakos}
\address{Department of Mathematics, University of Missouri, Columbia, MO 65211}
\email{grafakosl@missouri.edu}

\author[Nakamura]{Shohei Nakamura}
\address{Department of Mathematical Science and Information Science, Tokyo Metropolitan University, 1-1 Minami-Ohsawa, Hachioji, Tokyo, 192-0397, Japan}
\email{pokopoko9131@icloud.com}

\author[Nguyen]{Hanh Van Nguyen}
\address{Department of Mathematics, University of Alabama, Tuscaloosa, AL 35487}
%\address{Department of Mathematics, University of Missouri, Columbia, MO 65211}
%\address{Department of Mathematics, Hue College of Education, Hue, Vietnam}
\email{hvnguyen@ua.edu}

\author[Sawano]{Yoshihiro Sawano}
\address{Department of Mathematical Science and Information Science, Tokyo Metropolitan University, 1-1 Minami-Ohsawa, Hachioji, Tokyo, 192-0397, Japan}
\email{yoshihiro-sawano@celery.ocn.ne.jp}

\thanks{The first author would like to thank the Simons Foundation.}
\thanks{MSC 42B15, 42B30}%42B30 duplicated.
%Check here: http://www.ams.org/msc/msc2010.html?t=42B37

\begin{abstract}
We obtain boundedness from a product of Lebesgue or Hardy spaces into
Hardy spaces under suitable cancellation conditions for a large class of multilinear operators that includes the
Coifman-Meyer class, sums of products of linear Calder\'{o}n-Zygmund operators and combinations of these two types. 
\end{abstract}

\maketitle

\tableofcontents

%=============================================================%
\section{Introduction}
%=============================================================%

In this work, we obtain boundedness for 
multilinear singular operators of various types 
from products of Lebesgue or Hardy spaces into
Hardy spaces, under suitable cancellation conditions. 
This particular line of investigation 
 was initiated in the work of Coifman, Lions, Meyer and Semmes \cite{CLMS} 
who showed that certain bilinear operators with vanishing integral map $L^q\times L^{q'} $ 
into the Hardy space $H^1$ for $1<q<\infty$ with $q'=q/( q-1)$. 
This result was extended by Dobyinksi \cite{Do} for   Coifman-Meyer multiplier operators and by 
 Coifman and Grafakos \cite{CG} for finite sums of products of Calder\'{o}n-Zygmund operators. 
In \cite{CG} boundedness was extended to
 $H^{p_1}\times H^{p_2}\to H^p$ for the entire range 
 $0<p_1,p_2,p<\infty$ and $1/p=1/p_1+1/p_2$, under the necessary cancellation 
 conditions. 
 
 Additional proofs of these results were provided by Grafakos and Li \cite{GLLXW00}, Hu and Meng \cite{HuMe12}, and Huang and Liu \cite{HuangLiu13}. 
All the aforementioned accounts on this topic are based on different approaches and address two classes of 
 operators but \cite{CG}, \cite{HuMe12}, and \cite{HuangLiu13} seem to contain flaws in their proofs; 
in fact, as of this writing, only the approach in \cite{GLLXW00} stands, 
which deals with the case of finite sums of products of Calder\'{o}n-Zygmund operators. 
In this work we revisit this line of investigation 
via a new method based on $(p,\infty)$-atomic decompositions. 
Our approach is powerful enough to encompass many types of multilinear operators 
that include all the previously studied (Coifman-Meyer type and finite sums of products of Calder\'on-Zygmund operators) as well as mixed types. An alternative approach to Hardy space estimates for bilinear operators has    appeared in the recent work of
Hart and Lu \cite{HartLu}.

Recall that the Hardy space $H^p$ with $0<p<\infty$
is given as the space of all tempered distributions $f$ for which
\[
\|f\|_{H^p}=\big\|\sup_{t>0}|e^{t\Delta}f|\big\|_{L^p}
\]
is finite,
where $e^{t\Delta}$ denotes the heat semigroup
for $0<p \le \infty$.
Note that $H^p$ and $L^p$ are isomorphic with norm equivalence
when $1<p \le \infty$.

In this work we study the boundedness into $H^p$ of the following three types of operators:
\begin{itemize}
\item
multilinear singular integral operators of Coifman-Meyer type;
\item
sums of $m$-fold products of linear Calder\'{o}n-Zygmund singular integrals;
\item
multilinear singular integrals of mixed type (i.e., combinations of the previous two types).
\end{itemize}

Let $m,n$ be positive integers. 
For a bounded function $\sigma$ on $({\mathbb R}^n)^m$ 
we consider the multilinear operator 
\[
{\mathcal T}_\sigma(f_1,\ldots,f_m)(x)
=
\int_{({\mathbb R}^n)^m}
\sigma(\xi_1,\ldots,\xi_m)
\widehat{f_1}(\xi_1)\cdots\widehat{f_m}(\xi_m)
e^{2\pi  i x\cdot (\xi_1+\cdots+\xi_m)}
\,d\xi_1\cdots\,d\xi_m
\quad (x \in {\mathbb R}^n)
\]
for $f_1,\ldots,f_m \in {\mathscr S}$. 
Here $\mathscr S$ is the space of Schwartz functions 
and 
$\widehat f (\xi) = \int_{\mathbb R^n} f(x) e^{-2\pi i x\cdot \xi}dx$ 
is the Fourier transform of a given 
Schwartz function $f$ on $\mathbb R^n$. The space of tempered distributions is denoted by ${\mathscr S}'$.

Certain conditions on $\sigma$ imply that 
${\mathcal T}_\sigma$
extends to a bounded linear operator
from
$L^{p_1} \times \cdots \times L^{p_m}$
to
$L^p$
as long as 
$1<p_1,\ldots,p_m \le \infty$
and $0<p<\infty$
satisfies
\begin{equation}\label{Holder}
\frac{1}{p} = \frac{1}{p_1}+\cdots+\frac{1}{p_m}.
\end{equation}

Such a condition is the following by Coifman-Meyer (modeled after the classical Mihlin linear multiplier condition) 
\begin{equation}\label{CMcond}
|\partial^\alpha \sigma(\xi_1,\ldots,\xi_m)|
\lesssim
(|\xi_1|+\cdots+|\xi_m|)^{-|\alpha|},
\quad (\xi_1,\ldots,\xi_m) \in ({\mathbb R}^n)^m \setminus \{0\}
\end{equation}
for $\alpha \in ({\mathbb N}_0{}^n)^m$ satisfying 
$|\alpha| \le M$ for some large $M$.
Such operators are  called $m$-linear Calder\'{o}n-Zygmund operators and there is a rich theory for them analogous to the linear one.

An $m$-linear Calder\'{o}n-Zygmund operator associated 
with a Calder\'{o}n-Zygmund kernel $K$ on $\mathbb{R}^{mn}$ is defined by
\begin{equation}\label{eq.CalZygOPT}
 {\mathcal T}_\sigma(f_1,\ldots,f_m)(x) 
= \int_{({\mathbb R}^n)^m} K(x-y_1,\ldots,x-y_m)
f_1(y_1)\cdots f_m(y_m)\; dy_1\cdots dy_m,
 \end{equation}
where $\sigma$ is the distributional Fourier transform 
of $K$ on $({\mathbb R}^n)^m$
that satisfies (\ref{CMcond}). 
When $m=1$, these operators reduce to classical 
Calder\'{o}n-Zygmund singular integral operators. 

An $m$-linear operator of {\it product type}
on $\mathbb{R}^{mn}$ is defined by
\begin{equation}\label{eq.CalZygOPT-2}
\sum_{\rho=1}^T T_{\sigma_1^\rho}(f_1)(x)\cdots T_{\sigma_m^\rho}(f_m)(x)
\quad (x \in {\mathbb R}^n),
\end{equation}
where the $T_{\sigma_j^\rho}$'s are linear Calder\'{o}n-Zygmund operators 
associated with the multipliers $\sigma_j^\rho$. 
In terms of kernels these operators can be expressed as
\[
 {\mathcal T}_\sigma(f_1,\ldots,f_m)(x) =
\sum_{\rho=1}^T \prod_{j=1}^m \int_{{\mathbb R}^n} K^\rho_{\sigma_j}(x-y_j)
 f_j(y_j) dy_j, 
\]
where $K^\rho_1,\ldots,K^\rho_m$
are the Calder\'{o}n-Zygmund kernels
of the operator
$T^\rho_{\sigma_1},\ldots,T^\rho_{\sigma_m}$,
respectively
for $\rho=1,\ldots,T$.

In this work we also consider operators of {\it mixed type}, i.e., of the form
\begin{equation}\label{eq.CalZygOPT-3}
{\mathcal T}_\sigma(f_1,\ldots,f_m)(x)
=
\sum_{\rho=1}^T
\sum_{\substack{I_1^\rho,\ldots,I_{G(\rho)}^\rho
}}
\prod_{g=1}^{G(\rho)}
T_{\sigma_{I_g^\rho}}(\{f_l\}_{l \in I_g^{{\rho}} })(x),
\end{equation}
where for each $\rho=1,\ldots,T$, $I^\rho_1, \cdots ,I^\rho_{G(\rho)}$ is a partition of 
$\{1,\ldots,m\}$ 
and each $T_{\sigma_{I_g^\rho}}$ is an $| I_g^\rho|$-linear Coifman-Meyer multiplier operator.
We write $I^\rho_1+ \cdots + I^\rho_{G(\rho)}=\{1,\ldots,m\}$ to denote such partitions. 
 
In this work, we study operators of the form \eqref{eq.CalZygOPT}, \eqref{eq.CalZygOPT-2}, and 
\eqref{eq.CalZygOPT-3}. 
We will be working with indices in the following range 
\[
0<p_1,\ldots,p_m \le \infty, \quad
0<p< \infty
\] 
that satisfy (\ref{Holder}). Throughout this paper 
we reserve the letter $s$ to denote the following index:
\begin{equation}\label{eq:s}
s=[n(1/p-1)]_+
\end{equation}
and we fix $N\gg s$   a sufficiently large integer, say $N=m(n+1+2s)$.

We recall that a $(p,\infty)$-atom is an $L^\infty$-function
$a$ that satisfies $|a|\le \chi_Q$, 
where $Q$ is a cube on $\mathbb R^n$ with sides parallel to the axes and 
\[
\int_{{\mathbb R}^n}x^\alpha a(x)\,dx=0
\]
for all $\alpha$ with $|\alpha| \le N$.
By convention, when $p=\infty$, $a$ is called a $(\infty,\infty)$-atom 
 if    $Q= \mathbb{R}^n$ 
 and $\|a\|_{L^{\infty}} \le 1.$ No cancellations are required for $(\infty,\infty)$-atoms.

Our main results are as follows:
\begin{theorem}\label{ThmMain}
Let ${\mathcal T}_\sigma$ be the operator defined in \eqref{eq.CalZygOPT}  
and assume that it satisfies \eqref{CMcond}. Let 
$0<p_1,\ldots,p_m \le \infty$  and $0<p<\infty$ satisfy $(\ref{Holder})$.
 Assume that 
 \begin{equation}\label{eq.TmCan}
 \int_{{\mathbb R}^n} x^{\alpha}{\mathcal T}_\sigma(a_1,\ldots,a_m)(x)\; dx=0,
 \end{equation}
 for all $|\alpha|\le s$ and 
all $(p_l,\infty)$-atoms $a_l$.
Then ${\mathcal T}_\sigma$ can be extended to a bounded map from $H^{p_1}\times\cdots\times H^{p_m}$ to $H^p$. 
\end{theorem}

\begin{theorem}\label{ThmMain-2}
Let ${\mathcal T}_\sigma$ be the operator defined in \eqref{eq.CalZygOPT-2},
$0<p_1,\ldots,p_m< \infty$,
and
$0<p<\infty$ satisfies $(\ref{Holder})$, where each $\sigma_j^\rho$ satisfies \eqref{CMcond} with $m=1$.
 Assume that
$(\ref{eq.TmCan})$ holds
 for all $|\alpha|\le s$. Then ${\mathcal T}_\sigma$ can be extended to a bounded map from $H^{p_1}\times\cdots\times H^{p_m}$ to $H^p$.   
\end{theorem}

\begin{theorem}\label{ThmMain-3}
Let ${\mathcal T}_\sigma$ be the operator defined in \eqref{eq.CalZygOPT-3},
$0<p_1,\ldots,p_m \le \infty$,
and
$0<p<\infty$ satisfies $(\ref{Holder})$. 
Suppose that each $\sigma_{I_g^\rho}$ satisfies \eqref{CMcond} with $m=|I_g^\rho|$.
Assume that
$(\ref{eq.TmCan})$ holds
 for all $|\alpha|\le s$ and that
\begin{equation}\label{160902-1}
\sup_{\rho\, =1,\ldots,T}\sup_{I^\rho_1+ \cdots + I^\rho_{G(\rho)}=\{1,\ldots,m\}}
\inf_{l \in I_g^t}p_l<\infty. 
\end{equation}
Then ${\mathcal T}_\sigma$ can be extended to a bounded map 
from $H^{p_1}\times\cdots\times H^{p_m}$ to $H^p$.
\end{theorem}

\begin{remark}  
\noindent (1)
In Theorem \ref{ThmMain-2}, 
we exclude the case $p_l=\infty$ for all 
$l=1,\ldots, m$. 
In fact, one can not expect the mapping property of 
$\mathcal{T}_\sigma$ with \eqref{eq.CalZygOPT-2} if 
$p_l=\infty$ for some $l=1,\ldots,m$.
Similarly, in Theorem \ref{ThmMain-3}, 
we need to assume \eqref{160902-1} 
instead of the exclusion of the case $p_l=\infty$ for some $l=1,\ldots,m$.

 (2) The convergence of the integral in \eqref{eq.TmCan} is a consequence of Lemma \ref{lm.3A1} for all  $x$  
 outside the union of a fixed multiple of the supports of $a_i$, while the function $T(a_1, \dots, a_m)$ is 
 integrable for $x$ inside any compact set.
\end{remark}

A few comments about the notation. For brevity 
we write $d\vec{y}=dy_1\,\cdots\,dy_m$ 
and we use the symbol $C $ to denote a nonessential constant whose value 
may vary at different occurrences. 
For $(k_1,\ldots,k_m) \in {\mathbb Z}^m$,
we write $\vec{k}=(k_1,\ldots,k_m)$.
We use the notation $A\lesssim B$ to indicate that $A\le C\, B$ for some constant $C$. We denote the 
Hardy-Littlewood maximal operator by $M$:
\begin{equation}\label{eq:160918-1}
M f(x)=\sup_{r>0}\frac{1}{r^n}\int_{B(x,r)}|f(y)|\,dy.
\end{equation}
We say that $A\approx B$ if both $A\lesssim B$ and $B\lesssim A$ hold. 
The cardinality of a finite set $J$ is denoted by either $|J|$ or $\sharp J$. 

A cube $Q$ in $\mathbb R^n$ has sides parallel to the axes. 
We denote by $Q^*$ a centered-dilated cube of any cube $Q$ 
with the length scale factor $3\sqrt{n}$; then
\begin{equation}\label{eq:QStar}
Q^*=3\sqrt{n}Q^*, \quad
Q^{**}=9n Q.
\end{equation}

\section{Preliminary and related results}
\label{s:160725-2}

\subsection{Equivalent definitions of Hardy spaces}
We begin this section by recalling Hardy spaces.
Let $\phi \in C^\infty_{\rm c}$ satisfy
\begin{equation}\label{eq:160716-101}
{\rm supp}(\phi)
\subset
\left\{x\in \mathbb{R}^n\ :\ |x|\le 1\right\}
\end{equation} 
and 
\begin{equation}\label{eq:160716-102}
\int_{{\mathbb R}^n} \phi(y)\; dy=1.
\end{equation} 
For $t>0$, we set
$\phi_t(x)=t^{-n}\phi(t^{-1}x)$.
The maximal function $M_\phi$ associated 
with the smooth bump $\phi$ is given by:
\begin{equation}\label{eq:M phi}
M_\phi(f)(x) = \sup_{t>0}\big| (\phi_t*f)(x) \big| 
= \sup_{t>0}\Big| t^{-n}\int_{{\mathbb R}^n} \phi\big(y/t\big)f(x-y)\; dy \Big| 
\end{equation}
for $f\in \mathscr S'(\mathbb R^n)$. 
For $0<p<\infty$, 
the Hardy space $H^{p}$ is
characterized as the space of all tempered distributions $f$ for which $M_\phi(f)\in L^p$; also 
the $H^p$ quasinorm satisfies
\[
\|f\|_{H^p} \approx \|M_\phi(f)\|_{L^p}.
\]

Denote by $\mathscr C^\infty_{\rm c}$ the space of all smooth functions on $\mathbb{R}^n$ with compact support. 
The following density property of Hardy spaces will be useful in the proof of the main theorems. 

\begin{proposition}[{\rm \cite[Chapter III, 5.2(b)]{SteinHA}}]
\label{prop:160726-1}
Let $N \gg s$ be fixed.
Then the following space is dense in $H^p$:
\[
\mathcal{O}_N(\mathbb{R}^n)=
\bigcap_{\alpha \in {\mathbb N}_0^n, |\alpha| \le N}
\left\{f \in \mathscr C^\infty_{\rm c}\,:\,
\int_{{\mathbb R}^n}x^\alpha f(x)\,dx=0\right\},
\]
where $\mathscr C^\infty_{\rm c}$ is the space of all smooth functions with compact supports in $\mathbb{R}^n.$
\end{proposition}

The definition of the Hardy space is useful
as the following theorem implies:
\begin{theorem}[\cite{Nakai-Sawano-2014}]
\label{th-molecule}
Let $0<p<\infty$.

If $f \in H^p$, then
there exist a collection of $(p,\infty)$-atoms
$\{a_k\}_{k=1}^\infty$ and
a nonnegative sequence
$\{\lambda_k\}_{k=1}^\infty$
such that
\[
f=\sum_{k=1}^\infty \lambda_k a_k
\]
in ${\mathscr S}'({\mathbb R}^n)$
and that 
we have 
\begin{equation*}%\label{crucial77}
\Big\|
\sum_{k=1}^\infty
\lambda_k\chi_{Q_k}
\Big\|_{L^p}
\lesssim \|f\|_{H^p}.
\end{equation*}

Moreover, if $f \in \mathscr C^\infty_{\rm c}$ and
$\displaystyle
\int_{{\mathbb R}^n}x^\alpha f(x)\,dx=0
$
for all $\alpha$ with $|\alpha| \le [n(1/p-1)]_+$, 
then
we can arrange that $\lambda_k=0$ for all but finitely many $k$.
\end{theorem} 

The following lemma, whose proof is just an application of the Fefferman-Stein vector-valued inequality for maximal function, will be used frequently in the next sections.
\begin{lemma}
\label{lm.2B00}
If $\gamma >\max(1,\frac1p)$, $0<p<\infty$, $\lambda_k\ge 0$ and $\{Q_k\}_k$ are sequence of cubes, then
\[
\Big\|
\sum_{k}\lambda_k (M\chi_{Q_k})^{\gamma}
\Big\|_{L^p}
\lesssim
\Big\|\sum_{k}\lambda_k \chi_{Q_k}\Big\|_{L^p}.
\]
In particular
\[
\Big\|\sum_{k}\lambda_k \chi_{Q_k^{**}}\Big\|_{L^p}
\lesssim
\Big\|\sum_{k}\lambda_k \chi_{Q_k}\Big\|_{L^p}.
\]
\end{lemma}

We will also make use of the following result: 
\begin{lemma}\label{lm.2A00}
Let $p \in(0,\infty)$.
Assume that $q \in (p,\infty] \cap [1,\infty]$.
Suppose that we are given
a sequence of cubes $\{Q_j\}_{j=1}^\infty$
and a sequence of non-negative $L^q$-functions $\{F_j\}_{j=1}^\infty$.
Then
\[
\Big\|
\sum_{j=1}^\infty \chi_{Q_j}F_j
\Big\|_{L^p}
\lesssim
\Big\|
\sum_{j=1}^\infty 
\left(
\frac{1}{|Q_j|}
\int_{Q_j}F_j(y)^q\,dy\right)^{1/q}
\chi_{Q_j}
\Big\|_{L^p}.
\]
\end{lemma}

\begin{proof}
See
\cite{HuMe12}
for the case of $0<p \le 1$
and \cite{Nakai-Sawano-2014}, \cite{Sawano13}
for the case of $1<p<\infty$.
\end{proof}

\subsection{Reductions in the proof of main results}
To start the proof of the main results, let  $p_1,\ldots,p_m$ and $p$ be given as in  Theorems \ref{ThmMain}, \ref{ThmMain-2} or \ref{ThmMain-3} and note that 
$H^{p_l}\cap\mathcal{O}_N(\mathbb{R}^n)$ is dense in $H^{p_l}$ for $1\le l\le m$ and $0<p_l<\infty$. Recall the integer $N\gg s$ and fix 
$f_l\in H^{p_l}\cap\mathcal{O}_N(\mathbb{R}^n)$ for which $0<p_l<\infty$. By Theorem \ref{th-molecule}, 
we can decompose $f_l = \sum_{k_l =1}^\infty \lambda_{l,k_l} a_{l,k_l}$, where
$\{\lambda_{l,k_l}\}_{k_l=1}^\infty$ is a non-negative finite sequence
and
$\{a_{l,k_l}\}_{k_l=1}^\infty$ is a sequence of $(p_l,\infty)$-atoms
such that $a_{l,k_l}$ is supported in a cube $Q_{l,k_l}$ satisfying
\[
|a_{l,k_l}| \le \chi_{Q_{l,k_l}},\quad
\int_{{\mathbb R}^n}x^\alpha a_{l,k_l}(x)\,dx=0,\quad \ 
|\alpha| \le N
\]
and that
\begin{equation}\label{crucial77}
\Big\|%\left(
\sum_{k_l=1}^\infty %\left(
\lambda_{l,k_l}\chi_{Q_{l,k_l}}
%\right)^{u}
%\right)^{\frac1{u}}
\Big\|_{L^{p_l}}
\lesssim \|f_l\|_{H^{p_l}}.
\end{equation}
If $p_l=\infty$ and $f_l\in L^{\infty}$, then we can conventionally rewrite $f_l=\lambda_{l,k_l} a_{l,k_l}$ 
where $\lambda_{l,k_l}=\|f_l\|_{L^{\infty}}$ and $a_{l,k_l}=\|f\|_{L^{\infty}}^{-1}f$ is 
an $(\infty,\infty)$-atom supported in $Q_{l,k_l}=\mathbb{R}^n$. In this case the summation in \eqref{crucial77} is ignored since there is only one summand.

By the
multi-sublinearity of $M_\phi\circ \mathcal{T}_\sigma$, we can estimate
\[
M_\phi\circ \mathcal{T}_\sigma(f_1,\ldots,f_m)(x)\le 
\sum_{k_1,\ldots,k_m=1}^\infty
\Big(\prod_{l=1}^m \lambda_{l,k_l}\Big)
 M_\phi \circ {\mathcal T}_\sigma(a_{1,k_1},\ldots,a_{m,k_m})(x).
\]

To prove Theorems~\ref{ThmMain}, ~\ref{ThmMain-2}, and~\ref{ThmMain-3}, 
 it now suffices to establish the following result:
\begin{proposition}
\label{LM.Key-31}
 Let ${\mathcal T}_\sigma$ be the operator 
defined 
in \eqref{eq.CalZygOPT}, \eqref{eq.CalZygOPT-2} or \eqref{eq.CalZygOPT-3}. 
Let $p_1,\ldots,p_m$ and $p$ be given as in corresponding Theorems \ref{ThmMain}, \ref{ThmMain-2} or \ref{ThmMain-3}. 
Then we have 
 \begin{equation}\label{eq.PWEST}
\Big\|\sum_{k_1,\ldots,k_m=1}^\infty
\Big(\prod_{l=1}^m \lambda_{l,k_l}\Big)
 M_\phi \circ {\mathcal T}_\sigma(a_{1,k_1},\ldots,a_{m,k_m})\Big\|_{L^p}
\lesssim
\prod_{l=1}^m
\Big\|
\sum_{k_l=1}^\infty
\lambda_{l,k_l}\chi_{Q_{l,k_l}}
\Big\|_{L^{p_l}}.
\end{equation}
\end{proposition}

Notice that in view of 
\eqref{crucial77} and Proposition \ref{LM.Key-31}, 
one obtains the required estimate 
$$
\| {\mathcal T}_\sigma(f_1,\dots , f_m) \|_{H^p}
=
\|
M_\phi\circ \mathcal{T}_\sigma(f_1,\ldots,f_m)
\|_{L^p}
\lesssim \| f_1\|_{H^{p_1}} \cdots \| f_m\|_{H^{p_m}}\, . 
$$

We may therefore focus on the proof of Proposition~\ref{LM.Key-31}. 
In the sequel we will prove \eqref{eq.PWEST}. 
Its proof will depend on whether ${\mathcal T}_\sigma$
is of type \eqref{eq.CalZygOPT}, \eqref{eq.CalZygOPT-2} or \eqref{eq.CalZygOPT-3}. The detail proof for each  type is discussed in subsequent sections.

%=============================================================%
\section{The Coifman-Meyer type}
%=============================================================%

Throughout this section, $\mathcal{T}_\sigma$ denotes for the operator defined in  \eqref{eq.CalZygOPT}. The main purpose of this section is to establish \eqref{eq.PWEST} for  $\mathcal{T}_\sigma$.

\subsection{Fundamental estimates for the Coifman-Meyer type}

We treat the case of Coifman-Meyer multiplier operators whose symbols satisfy \eqref{CMcond}. 
The study of such operators was initiated by Coifman and Meyer \cite{CM2}, \cite{CM3} 
and was later pursued by Grafakos and Torres \cite{GrafakosTorresAdvances}; see also \cite{G} for an account.
Denoting by $K$ the inverse Fourier transform of $\sigma$, 
in view of \eqref{CMcond},
we have
\[
|\partial^{\beta}_{y}
K(y_1,\ldots, y_m)|
\lesssim
\big( \sum_{i=1}^m |y_i| \big)^{-mn -|\beta|},\quad (y_1,\ldots,,y_m)\ne (0,\ldots,0)
\]
for all 
$\beta=(\beta_1,\ldots,\beta_m)\in\N_0{}^{mn}=(\N_0{}^n)^m$
and $|\beta|\le N$.

Examining carefully the smoothness of the kernel,
we obtain the following estimates:

\begin{lemma}
\label{lm.3A1}
Let $a_k$ be $(p_k,\infty)$-atoms supported in $Q_k$ for all $1\le k\le m$. Let $\Lambda$ be a non-empty subset of $\{1,\ldots,m\}$. Then we have
\[
|\mathcal{T}_\sigma(a_1,\ldots,a_m)(y)|
\le
\dfrac{\min\{\ell(Q_k) : k\in \Lambda\}^{n+N+1}}
{\big(\sum_{k\in \Lambda}|y-c_k|\big)^{n+N+1}}
\]
for all $y\notin \cup_{k\in \Lambda}Q_k^*$.
\end{lemma}
\begin{proof}
We may suppose that $\Lambda=\{1,\ldots,r\}$ for some $1\le r\le m$ and that
\[
\ell(Q_1) = \min\{\ell(Q_k) : k\in \Lambda\}.
\]
Let $c_k$ be the center of $Q_k$ and fix $y\notin \cup_{k\in \Lambda}Q_k^*$. Using the cancellation of $a_1$ we can rewrite
\begin{align}
\notag
\mathcal{T}_\sigma(a_1,\ldots,a_m)(y) =&
 \int_{\mathbb{R}^{mn}}K(y-y_1,\ldots,y-y_m)a_1(y_1)\cdots a_m(y_m)d\vec{y}\\
\notag  
 =&
 \int_{\mathbb{R}^{mn}}\big[K(y-y_1,\ldots,y-y_m)-P_N(y,y_1,y_2,\ldots,y_m)\big]
 a_1(y_1)\cdots a_m(y_m)d\vec{y}\\
 =&
 \int_{\mathbb{R}^{mn}}K^1(y,y_1,y_2,\ldots,y_m)
 a_1(y_1)\cdots a_m(y_m)d\vec{y},  \label{eq.3A5}
\end{align}
where 
\begin{equation*}
P_N(y,y_1,y_2,\ldots,y_m) = \sum_{|\alpha|\le N}\frac{1}{\alpha!}
\partial^{\alpha}_{1}K(y-c_1,y-y_2,\ldots,y-y_m)(c_1-y_1)^{\alpha}
\end{equation*}
is the Taylor polynomial of degree $N$ of $K(y-\cdot,y-y_2,\ldots,y-y_m)$ at $c_1$ and
\begin{equation}
\label{eq.3A6}
K^1(y,y_1,\ldots,y_m) = K(y-y_1,\ldots,y-y_m)-P_N(y,y_1,y_2,\ldots,y_m).
\end{equation}

By the smoothness condition of the kernel and the fact that
\[
|y-y_k|\approx |y-c_k| %,  \quad\forall k\in ,
\]
for all $k \in\Lambda$ and $y_k\in Q_k$
we can estimate
\begin{align*}
\big|K(y,y_1,\ldots,y_m)-P_N(y,c_1,y_2,\ldots,y_m)\big|\lesssim&
|y_1-c_1|^{N+1}\Big(
\sum_{k\in \Lambda}|y-c_k|+\sum_{j=2}^m|y-y_j|
\Big)^{-mn-N-1}.
\end{align*}
Thus,
\begin{align*}
|\mathcal{T}_\sigma(a_1,\ldots,a_m)(y)| \lesssim&
\int_{\mathbb{R}^{mn}}\frac{|y_1-c_1|^{N+1}|a_1(y_1)|\cdots |a_m(y_m)|}
 {\Big(
\sum_{k\in \Lambda}|y-c_k|+\sum_{j=2}^m|y-y_j|
\Big)^{mn+N+1}}
 d\vec{y}\\
 \lesssim&
 \int_{\mathbb{R}^{(m-1)n}}\frac{
 \ell(Q_1)^{n+N+1}}
 {\Big(
\sum_{k\in \Lambda}|y-c_k|+\sum_{j=2}^m|y_j|
\Big)^{mn+N+1}}
 dy_2\cdots dy_m\\
 \lesssim&
 \frac{\ell(Q_1)^{n+N+1}}
 {\Big(
\sum_{k\in \Lambda}|y-c_k|
\Big)^{n+N+1}}.
\end{align*}
\end{proof}

%%%%
\begin{lemma}\label{lem:160726-51}
Let $a_k$ be $(p_k,\infty)$-atoms supported in $Q_k$ for all $1\le k\le m$.
Suppose $Q_1$ is the cube such that $\ell(Q_1) =  \min\{\ell(Q_k) : 1\le k\le m\}$.
Then for fixed $1\le r<\infty$ and $j\in \mathbb{N}$, we have
\begin{align}
\label{eq.3A2}
\|\mathcal{T}_\sigma(a_1,\ldots,a_m)\chi_{Q_1^{**}}\|_{L^{r}}&\lesssim
|Q_1|^{\frac{1}{r}}
\prod_{l=1}^m
\inf_{z\in Q_1^{*}}
M\chi_{Q_l}(z)^\frac{n+N+1}{mn},\\
\label{eq.3A3}
\|M\circ \mathcal{T}_\sigma(a_1,\ldots,a_m)\chi_{Q_1^{**}}\|_{L^{r}}&\lesssim
|Q_1|^{\frac{1}{r}}
\prod_{l=1}^m
\inf_{z\in Q_1^{*}}
M\chi_{Q_l}(z)^\frac{n+N+1}{mn},
\end{align}
%
%\begin{equation}
%\label{eq.3D2}
%\|\mathcal{T}_\sigma(a_1,\ldots,a_m)\|_{\rm BMO}
%\lesssim
%\prod_{l=1}^m
%\inf_{z\in 2^jQ_1^{*}}
%M\chi_{2^jQ_l^{**}}(z)^\frac{n+N+1}{mn}.
%\end{equation}
Furthermore, 
if $Q_0$ is a cube such 
that $\ell(Q_0)\le \ell(Q_1)$  and $2^jQ_0^{**}\cap 2^jQ_l^{**}=\emptyset$ for some $l$, then
\begin{equation}\label{eq.3D3}
\|\mathcal{T}_\sigma(a_1,\ldots,a_m)\chi_{2^jQ_0^{**}}\|_{L^\infty}
\lesssim
\prod_{l=1}^m
\inf_{z\in 2^jQ_0^{*}}
M\chi_{2^jQ_l^{**}}(z)^\frac{n+N+1}{mn}.
\end{equation}
In particular, 
under the above assumption, 
\begin{align}\label{eq.3D4}
\Big(
\frac{1}{|2^jQ_0^{**}|}
\int_{2^jQ_0^{**}}
|\mathcal{T}_\sigma(a_1,\ldots,a_m)(y)|^rdy
\Big)^{\frac1r}
&\lesssim
\prod_{l=1}^m
\inf_{z\in 2^jQ_0^{*}}
M\chi_{2^jQ_l^{**}}(z)^\frac{n+N+1}{mn}.
\end{align}
\end{lemma}
\begin{proof}
To check \eqref{eq.3A2}, it is enough to consider $1<r<\infty$ and two following cases. First, if $Q_1^{**}\cap Q_k^{**}\ne \emptyset$ for all $2\le k\le m$, then, by the assumption $\ell(Q_1) = \min\{\ell(Q_k) : 1\le k\le m\}$, $Q_1^{**}\subset 3Q_k^{**}$ for all $1\le k\le m$. This implies
\[
\inf_{z\in Q_1^{*}}M\chi_{3Q_k^{**}}(z)\ge 1,
\]
for all $1\le k\le m.$
Now the boundedness of $\mathcal{T}_\sigma$ from $L^r\times L^\infty\times\cdots\times L^\infty$ to $L^r$ yields
\begin{align}
\notag
\|\mathcal{T}_\sigma(a_1,\ldots,a_m)\chi_{Q_1^{**}}\|_{L^r}
\le&
\|\mathcal{T}_\sigma(a_1,\ldots,a_m)\|_{L^r}\\
\notag
\lesssim&
\|a_1\|_{L^r}\|a_2\|_{L^\infty}\cdots \|a_m\|_{L^\infty}\\
\label{eq.3A4}
\lesssim&
|Q_1|^{\frac{1}{r}}
\prod_{k=1}^m\inf_{z\in Q_1^{*}}M\chi_{3Q_k^{**}}(z)^\frac{n+N+1}{mn}.
\end{align}

Second, if $Q_1^{**}\cap Q_k^{**}=\emptyset$ for some $k$, then the set
\[
\Lambda = \{2\le k\le m : Q_1^{**}\cap Q_k^{**}=\emptyset\}
\]
is a non-empty subset of $\{1,\ldots, m\}$. Fix arbitrarily $y\in \mathbb{R}^n$.  By the cancellation of $a_1$, 
rewrite
\[
\mathcal{T}_\sigma(a_1,\ldots,a_m)(y)=
 \int_{\mathbb{R}^{mn}}K^1(y,y_1,y_2,\ldots,y_m)
 a_1(y_1)\cdots a_m(y_m)d\vec{y},
\]
where $K^1(y,y_1,\ldots,y_m)$ is defined in \eqref{eq.3A6}.
For $y_1\in Q_1$ we estimate
\begin{align*}
\big|K^1(y,y_1,\ldots,y_m)\big|\le&
C\ell(Q_1)^{N+1}\Big(
|y-\xi_1|+\sum_{j=2}^m|y-y_j|
\Big)^{-mn-N-1},
\end{align*}
for some $\xi_1\in Q_1$ and for all $y_l\in Q_l$.

Since $Q_1^{**}\cap Q_k^{**}=\emptyset$ for all $k\in \Lambda$, $|y-\xi_1|+|y-y_k|\ge |\xi_1-y_k| \ge C|c_1-c_k|$ for all $y_k\in Q_k$ and  $k\in \Lambda$.
Therefore
\[
\big|K^1(y,y_1,\ldots,y_m)\big|\lesssim
\ell(Q_1)^{N+1}\Big(
\sum_{k\in \Lambda}|c_1-c_k|
+\sum_{j=2}^{m}|y-y_j|
\Big)^{-mn-N-1},
\]
for all $y_1\in Q_1^*$ and $y_k\in Q_k$ for $k\in \Lambda$. Insert the above inequality into \eqref{eq.3A5} to obtain
\[
|\mathcal{T}_\sigma(a_1,\ldots,a_m)(y)|
\lesssim
\dfrac{\ell(Q_1)^{n+N+1}}
{\big(\sum_{k\in \Lambda}|c_1-c_k|\big)^{n+N+1}}
\lesssim
\dfrac{\ell(Q_1)^{n+N+1}}
{\sum_{k\in \Lambda}\big[\ell(Q_1)+|c_1-c_k|+\ell(Q_k)\big]^{n+N+1}}.
\]
Noting that 
$Q_1^{**}\subset 3Q_l^{**}$
for $l\notin \Lambda,$
the last inequality gives  
\begin{equation}
\label{eq.3A8}
\|\mathcal{T}_\sigma(a_1,\ldots,a_m)\|_{L^\infty}
\lesssim
\prod_{k=1}^m\inf_{z\in Q_1^{*}}M\chi_{3Q_k^{**}}(z)^\frac{n+N+1}{mn},
\end{equation}
which yields
\begin{equation}
\label{eq.3A9}
\|\mathcal{T}_\sigma(a_1,\ldots,a_m)\chi_{Q_1^{**}}\|_{L^r}
\lesssim
|Q_1|^{\frac{1}{r}}
\prod_{k=1}^m\inf_{z\in Q_1^{*}}M\chi_{3Q_k^{**}}(z)^\frac{n+N+1}{mn}.
\end{equation}
Combining \eqref{eq.3A4} and \eqref{eq.3A9} and noting that $M\chi_{3Q} \lesssim M\chi_Q$, we obtain \eqref{eq.3A2}.

Similarly, we can prove \eqref{eq.3A3}--\eqref{eq.3D3}. 
For example, 
to show \eqref{eq.3A3}, we again consider the case where 
$Q_1^{**}\cap Q_l^{**}\neq \emptyset$ holds for all $l$ and the case where this fails. 
In the first case, using the boundedness of $M$ on $L^r$, we arrive at the same situation as above. 
In the second case, 
we use the boundedness of $M$ on $L^\infty$ to see 
\[
\|M\circ \mathcal{T}_\sigma(a_1,\dots,a_m)\chi_{Q_1^{**}}\|_{L^r}
\lesssim 
|Q_1|^\frac{1}{r}
\|
\mathcal{T}_\sigma(a_1,\dots,a_m)\|_{L^\infty}.
\] 
Notice that the right-hand side is already treated in \eqref{eq.3A8}.
\end{proof}

Lemma \ref{lem:160726-51} will be used to study the behavior of the operator $M_\phi\circ \mathcal{T}_\sigma$ inside $Q_1^{**}$. For the region outside of $Q_1^{**}$, we need the following estimates.
\begin{lemma}
\label{lm.3A10}
Let $a_k$ be $(p_k,\infty)$-atoms supported in $Q_k$ for all $1\le k\le m$.
If $p_k=\infty$ then $Q_k=\mathbb{R}^n$. Suppose that $Q_1$ is the cube for which $\ell(Q_1) =  \min\{\ell(Q_k) : 1\le k\le m\}$.
Fix $0<t<\infty$.
\begin{enumerate}
\item
If $x \notin Q_1^{**}$ and $c_1 \notin B(x,100n^2t)$,
then
\begin{equation}\label{eq.3A10}
\frac{1}{t^n}
\int_{B(x,t)}|{\mathcal T}_\sigma(a_1,\ldots,a_m)(y)|\,dy
\lesssim \prod_{l=1}^m 
M\chi_{Q_l}(x)^{\frac{n+N+1}{m n}}.
\end{equation}
\item
If $x \notin Q_1^{**}$ and $c_1 \in B(x,100n^2t)$, then
\begin{equation}\label{eq.3A11}
\frac{\ell(Q_1)^{s+1}}{t^{n+s+1}}
\int_{Q_1^*}|{\mathcal T}_\sigma(a_1,\ldots,a_m)(y)|\,dy
\lesssim
M\chi_{Q_1}(x)^{\frac{n+s+1}{n}}
\prod_{l=1}^m 
\inf_{z \in Q_1^*}
M\chi_{Q_l}(z)^{\frac{n+N+1}{m n}},
\end{equation}
and
\begin{equation}\label{eq.3A12}
\frac{1}{t^{n+s+1}}
\int_{(Q_1^*)^c}
|y-c_1|^{s+1}|{\mathcal T}_\sigma(a_1,\ldots,a_m)(y)|\,dy
\lesssim 
M\chi_{Q_1}(x)^{\frac{n+s+1}{n}}
\prod_{l=1}^m 
\inf_{z \in Q_1^*}
M\chi_{Q_l}(z)^{\frac{N-s}{m n}}.
\end{equation}
\item
For all $x\notin Q_1^{**}$, we have
\begin{align}
\label{eq.3B12}
M_\phi\circ \mathcal{T}_\sigma(a_1,\ldots,a_m)(x)\lesssim
 \prod_{l=1}^m 
M\chi_{Q_l}(x)^{\frac{n+N+1}{m n}}
+
M\chi_{Q_1}(x)^{\frac{n+s+1}{n}}
\prod_{l=1}^m 
\inf_{z \in Q_1^*}
\Big(
M\chi_{Q_l}(z)^{\frac{N-s}{m n}}
\Big).
\end{align}
\end{enumerate}
\end{lemma}
\begin{proof}
Fix $x\notin Q_1^{**}$ and denote $\Lambda = \{1\le k\le m : x\notin Q_k^{**}\}$.

\textit{(1)} Suppose $c_1 \notin B(x,100n^2t)$. For $y\in B(x,t)$, from \eqref{eq.3A5} we rewrite
\[
\mathcal{T}_\sigma(a_1,\ldots,a_m)(y) = \int_{\mathbb{R}^{mn}}K^1(y,y_1,\ldots,y_m)
a_1(y_1)\cdots a_m(y_m)d\vec{y},
\]
where $K^1$ is defined in \eqref{eq.3A6}. 
Note that for $y\in B(x,t)$, $y_1\in Q_1$ and $c_1\notin B(x,100n^2t)$, we have
\[
t\lesssim |x-c_1|\lesssim |y-y_1|.
\]
Since $x\notin Q_k^{**}$ for all $k\in \Lambda$,
\[
|x-c_k|\lesssim |x-y_k|\lesssim t+ |y-y_k|\lesssim |y-y_1|+ |y-y_k|
\]
for all $k \in \Lambda$ and $y_k \in Q_k$.
Consequently,
\begin{align}
\label{eq.3A13}
\left|K^{1}(y,y_1,\ldots,y_m)\prod_{l=1}^m a_l(y_l)\right|
\lesssim
\frac{\ell(Q_1)^{N+1}\chi_{Q_1}(y_1)}{\displaystyle
\left(\sum_{l=2}^m |y-y_l|+
\sum_{k \in \Lambda} |x-c_k|\right)^{m n+N+1}}.
\end{align}
Integrating
(\ref{eq.3A13})
over $({\mathbb R}^n)^m$, 
and using that $\ell(Q_1)\leq \ell(Q_l)$ for all $2\le l\le m$, we obtain that 
\begin{align*}
|{\mathcal T}_\sigma(a_1,\ldots,a_m)(y)|
&\lesssim
\frac{\ell(Q_1)^{n+N+1}}{\displaystyle
\left(\sum_{l \in \Lambda} |x-c_l|\right)^{n+N+1}}
\\
&\lesssim
\prod_{l\in \Lambda}
\frac{\ell(Q_l)^{\frac{n+N+1}{|\Lambda|}}}{\displaystyle
|x-c_l|^{\frac{n+N+1}{|\Lambda|}}}
\chi_{(Q_l^{**})^c}(x)
\cdot
\prod_{k\notin \Lambda}\chi_{Q_k^{**}}(x)
\lesssim
\prod_{l=1}^m M\chi_{Q_l}(x)^{\frac{n+N+1}{m n}}.
\end{align*}
This pointwise estimate proves \eqref{eq.3A10}.

\textit{(2)} Assume $c_1 \in B(x,100n^2t)$. Fix $1<r<\infty$ and estimate the left-hand side of \eqref{eq.3A11} by
\[
\frac{\ell(Q_1)^{s+1}}{t^{n+s+1}}|Q_1|^{1-\frac1{r}}
\|\mathcal{T}_\sigma(a_1,\ldots,a_m)\chi_{Q_1^{**}}\|_{L^r}
\lesssim
\frac{\ell(Q_1)^{n+s+1}}{t^{n+s+1}}
\prod_{l=1}^m 
\inf_{z \in Q_1^*}
M\chi_{Q_l}(z)^{\frac{n+N+1}{m n}},
\]
where we used \eqref{eq.3A2} in the above inequality.
Since $x\notin Q_1^{**}$ and $c_1\in B(x,100n^2t)$, $Q_1^*\subset B(x,1000n^2t)$ and hence, $\ell(Q_1)/t\lesssim M\chi_{Q_1}(x)$. This combined with the last inequality implies \eqref{eq.3A11}.

To verify \eqref{eq.3A12}, we recall
the expression of $\mathcal{T}_\sigma(a_1,\ldots,a_m)(y)$ in \eqref{eq.3A5} 
and the pointwise estimate for $K^1(y,y_1,\ldots,y_m)$ defined in \eqref{eq.3A6}.
Denote $J=\{2\le k\le m : Q_1^{**}\cap Q_k^{**}=\emptyset\}$.
Using the facts that 
$|y-y_1|\sim |y-c_1|\geq \ell(Q_1)$ for $y\notin Q_1^*$, $y_1\in Q_1$ 
and
$|y-y_1|+|y-y_l|\geq |y_1-y_l|\gtrsim |z-c_l|$ 
for all $z\in Q_1^*$ and $l\in J$, we now estimate
\begin{align*}
|{\mathcal T}_\sigma(a_1,\ldots,a_m)(y)|
\lesssim
\int_{({\mathbb R}^n)^m}
\frac{\ell(Q_1)^{N+1}\chi_{Q_1}(y_1)\,d\vec{y}}{
\left(
\ell(Q_1)+|y-c_1|+
\displaystyle
\sum_{l\in J}
|z-c_l|
+\sum_{l=2}^m|y-y_l|
\right)^{m n+N+1}}
\end{align*}
for all $y \in (Q_1^*)^c$ and $z\in Q_1^*$.
Thus, 
\begin{align*}
&\frac{1}{t^{n+s+1}}
\int_{(Q_1^*)^c}|y-c_1|^{s+1}|{\mathcal T}_\sigma(a_1,\ldots,a_m)(y)|\,dy
\\
&\lesssim
\frac{1}{t^{n+s+1}}
\int_{{\mathbb R}^n \times ({\mathbb R}^n)^m}
\frac{|y-c_1|^{s+1}\ell(Q_1)^{N+1}\chi_{Q_1}(y_1)\,d\vec{y}dy}
{\displaystyle\bigg(\ell(Q_1)+|y-c_1|
+\sum_{l\in J}|c_1-c_l|
+\sum_{l=2}^m |y-y_l|\bigg)^{m n+N+1}}
\\
&\lesssim 
\left(
\frac{\ell(Q_1)}{t}
\right)^{n+s+1}
\prod_{l\in J}
\left(
\frac{\ell(Q_l)}{|z-c_l|}
\right)^\frac{N-s}{m}.
\end{align*}
Note that
$
1\lesssim\inf_{z\in Q_1^*}M\chi_{2Q_l^{**}}(z)
$
if $Q_1^{**} \cap Q_l^{**} \ne \emptyset$;
otherwise,
$\ell(Q_l)/|z-c_l|\lesssim M\chi_{2Q_l^{**}}(z)^\frac{1}{n}$
 for all  $z\in Q_1^*$.
Consequently,
\begin{align*}
\frac{1}{t^{n+s+1}}
\int_{(Q_1^*)^c}|y-c_1|^{s+1}|{\mathcal T}_\sigma(a_1,\ldots,a_m)(y)|\,dy
&\lesssim
M\chi_{Q_1}(x)^{\frac{n+s+1}{n}}
\prod_{l=1}^m
\inf_{z \in Q_1^*} M\chi_{Q_l}(z)^{\frac{N-s}{mn}},
\end{align*}
which deduces \eqref{eq.3A12}.

\textit{(3)} It remains to prove \eqref{eq.3B12}. Fix $x\notin Q_1^{**}$.
To calculate 
$M_{\phi}\circ \mathcal{T}_{\sigma}(a_1,\ldots,a_m)(x)$, 
we need to estimate
\[
\Big|
\int_{{\mathbb R}^n} \phi_t(x-y){\mathcal T}_\sigma(a_1,\ldots,a_m)(y)\,dy
\Big|
\]
for each $t\in (0,\infty)$. Let consider two cases: $c_1 \notin B(x,100n^2t)$ and
 $c_1 \in B(x,100n^2t)$.

In the first case,
since $\phi$ is supported in the unit ball,
\begin{equation*}
\Big|
\int_{{\mathbb R}^n} \phi_t(x-y){\mathcal T}_\sigma(a_1,\ldots,a_m)(y)\,dy
\Big|
\lesssim
\frac{1}{t^n}
\int_{B(x,t)}|{\mathcal T}_\sigma(a_1,\ldots,a_m)(y)|\,dy.
\end{equation*}
Since $c_1 \notin B(x,100n^2t)$, \eqref{eq.3A10} implies that
\begin{equation}\label{eq.3A14}
\Big|
\int_{{\mathbb R}^n} \phi_t(x-y){\mathcal T}_\sigma(a_1,\ldots,a_m)(y)\,dy
\Big|
\lesssim
\prod_{l=1}^m
M\chi_{Q_l}(x)^{\frac{n+N+1}{m n}}.
\end{equation}

In the second case,
we will exploit the moment condition
of ${\mathcal T}_\sigma(a_1,\ldots,a_m)$.
Denote
\begin{equation}
\label{eq.3C01}
\delta_1^s(t;x,y) =
\phi_t(x-y) -
\sum_{|\alpha|\le s}\frac{\partial^{\alpha}[\phi_t](x-c_1)}{\alpha!}(c_1-y)^{\alpha}.
\end{equation}
Since $|\delta_1^s(t;x,y)|\lesssim t^{-n-s-1}$ for all $x,y$ and \eqref{eq.TmCan},
\begin{align}
\notag
\Big|
\int_{{\mathbb R}^n} \phi_t(x-y){\mathcal T}_\sigma(a_1,\ldots,a_m)(y)\,dy
\Big|&=
\Big|
\int_{{\mathbb R}^n} \delta^s_1(t;x,y){\mathcal T}_\sigma(a_1,\ldots,a_m)(y)\,dy
\Big|
\\
\notag
\lesssim&
\frac{1}{t^{n+s+1}}
\int_{{\mathbb R}^n}|y-c_1|^{s+1}|{\mathcal T}_\sigma(a_1,\ldots,a_m)(y)|\,dy
\\
\notag
=&
\frac{1}{t^{n+s+1}}
\int_{Q_1^*}|y-c_1|^{s+1}|{\mathcal T}_\sigma(a_1,\ldots,a_m)(y)|\,dy\\
\notag
&+
\frac{1}{t^{n+s+1}}
\int_{(Q_1^*)^c}|y-c_1|^{s+1}|{\mathcal T}_\sigma(a_1,\ldots,a_m)(y)|\,dy\\
&\lesssim
\frac{\ell(Q_1)^{s+1}}{t^{n+s+1}}
\int_{Q_1^*}|{\mathcal T}_\sigma(a_1,\ldots,a_m)(y)|\,dy
\label{eq.3C02}\\
\notag
&+
\frac{1}{t^{n+s+1}}
\int_{(Q_1^*)^c}|y-c_1|^{s+1}|{\mathcal T}_\sigma(a_1,\ldots,a_m)(y)|\,dy.
\end{align}
Invoking \eqref{eq.3A11} and \eqref{eq.3A12}, we obtain
\begin{align}
\notag
&
\Big|
\int_{{\mathbb R}^n} \phi_t(x-y){\mathcal T}_\sigma(a_1,\ldots,a_m)(y)\,dy
\Big|\\
\notag
&\lesssim
M\chi_{Q_1}(x)^{\frac{n+s+1}{n}}
\prod_{l=1}^m 
\inf_{z \in Q_1^*}
\Big[
M\chi_{Q_l}(z)^{\frac{n+N+1}{m n}}
+
M\chi_{Q_l}(z)^{\frac{N-s}{m n}}
\Big]\\
\label{eq.3A15}
&\lesssim
M\chi_{Q_1}(x)^{\frac{n+s+1}{n}}
\prod_{l=1}^m 
\inf_{z \in Q_1^*}
\Big(
M\chi_{Q_l}(z)^{\frac{N-s}{m n}}
\Big).
\end{align}
Combining \eqref{eq.3A14} and \eqref{eq.3A15} yields the required estimate \eqref{eq.3B12}.
The proof of Lemma \ref{lm.3A10} is now completed.
\end{proof}

\subsection{The proof of Proposition \ref{LM.Key-31} for Coifman-Meyer type}
We now turn into the proof of \eqref{eq.PWEST}, i.e., estimate
\begin{equation}
\label{eq.3C3}
A= \Big\|\sum_{k_1,\ldots,k_m=1}^\infty
\Big(\prod_{l=1}^m \lambda_{l,k_l}\Big)
 M_\phi \circ {\mathcal T}_\sigma(a_{1,k_1},\ldots,a_{m,k_m})\Big\|_{L^p} .
\end{equation}
For each $\vec{k}=(k_1,\ldots,k_m)$, we denote by $R_{\vec{k}}$ the cube with smallest length among $Q_{1,k_1},\ldots,Q_{m,k_m}$. Then we have $A\lesssim B+G$, where
\begin{equation}
\label{eq.3C4}
B = \Big\|\sum_{k_1,\ldots,k_m=1}^\infty
\Big(\prod_{l=1}^m \lambda_{l,k_l}\Big)
 M_\phi \circ {\mathcal T}_\sigma(a_{1,k_1},\ldots,a_{m,k_m})\chi_{R_{\vec{k}}^{**}}\Big\|_{L^p}
\end{equation}
and
\begin{equation}
\label{eq.3C5}
G = \Big\|\sum_{k_1,\ldots,k_m=1}^\infty
\Big(\prod_{l=1}^m \lambda_{l,k_l}\Big)
 M_\phi \circ {\mathcal T}_\sigma(a_{1,k_1},\ldots,a_{m,k_m})\chi_{(R_{\vec{k}}^{**})^c}\Big\|_{L^p}.
\end{equation}

To estimate $B$, for some $\max(1,p)<r<\infty$ Lemma \ref{lm.2A00} and \eqref{eq.3A3} imply
\begin{align*}
B\lesssim &
\Big\|\sum_{k_1,\ldots,k_m=1}^\infty
\Big(\prod_{l=1}^m \lambda_{l,k_l}\Big)
\frac{\chi_{R_{\vec{k}}^{**}}}{|\chi_{R_{\vec{k}}^{**}}|^{\frac1r}}
\|
M_\phi \circ {\mathcal T}_\sigma(a_{1,k_1},\ldots,a_{m,k_m})\chi_{R_{\vec{k}}^{**}}
\|_{L^r}
 \Big\|_{L^p}
 \\
 \lesssim&
\Big\|\sum_{k_1,\ldots,k_m=1}^\infty
\Big(\prod_{l=1}^m \lambda_{l,k_l}\Big)
\Big(
\prod_{l=1}^m
\inf_{z\in R_{\vec{k}}^{*}}
M\chi_{Q_{l,k_l}}(z)^{\frac{n+N+1}{mn}}
\Big)
\chi_{R_{\vec{k}}^{**}}
 \Big\|_{L^p}\\
=&
\Big\|\sum_{k_1,\ldots,k_m=1}^\infty
\Big(
\prod_{l=1}^m
\inf_{z\in R_{\vec{k}}^{*}}
\lambda_{l,k_l}
M\chi_{Q_{l,k_l}}(z)^{\frac{n+N+1}{mn}}
\Big)
\chi_{R_{\vec{k}}^{**}}
 \Big\|_{L^p}\\
 \lesssim&
\Big\|\sum_{k_1,\ldots,k_m=1}^\infty
\Big(
\prod_{l=1}^m
\inf_{z\in R_{\vec{k}}^{*}}
\lambda_{l,k_l}
M\chi_{Q_{l,k_l}}(z)^{\frac{n+N+1}{mn}}
\Big)
\chi_{R_{\vec{k}}^{*}}
 \Big\|_{L^p},
\end{align*} 
where we used Lemma \ref{lm.2B00} in the last inequality. Now we can remove the infimum and apply   H\"older's inequality to obtain
\begin{align}
\notag
B\lesssim &
\Big\|\sum_{k_1,\ldots,k_m=1}^\infty
\prod_{l=1}^m
\lambda_{l,k_l}
\Big(M\chi_{Q_{l,k_l}}\Big)^{\frac{n+N+1}{mn}}
 \Big\|_{L^p}
 =
\Big\|
\prod_{l=1}^m
\sum_{k_l=1}^\infty
\lambda_{l,k_l}
\Big(
M\chi_{Q_{l,k_l}}\Big)^{\frac{n+N+1}{mn}}
\Big\|_{L^p}\\
\notag
\le&
\prod_{l=1}^m
\Big\|
\sum_{k_l=1}^\infty
\lambda_{l,k_l}
\Big(
M\chi_{Q_{l,k_l}}\Big)^{\frac{n+N+1}{mn}}
\Big\|_{L^{p_l}}
\lesssim
\prod_{l=1}^m
\Big\|
\sum_{k_l=1}^\infty
\lambda_{l,k_l}
\chi_{Q_{l,k_l}^{**}}
\Big\|_{L^{p_l}}\\
\label{eq.3A16}
\lesssim&
\prod_{l=1}^m
\Big\|
\sum_{k_l=1}^\infty
\lambda_{l,k_l}
\chi_{Q_{l,k_l}}
\Big\|_{L^{p_l}}.
\end{align}
Once again, Lemma \ref{lm.2B00} was used in the last two inequalities.

To deal with $G$, we use \eqref{eq.3B12} and estimate $G \lesssim G_1+G_2$, where
\[
G_1  = 
\Big\|\sum_{k_1,\ldots,k_m=1}^\infty
\Big(\prod_{l=1}^m \lambda_{l,k_l}\Big)
  \prod_{l=1}^m 
\left(M\chi_{Q_{l,k_l}}\right)^{\frac{n+N+1}{m n}}
 \Big\|_{L^p}
\]
and
\[
G_2  = 
\Big\|\sum_{k_1,\ldots,k_m=1}^\infty
\Big(\prod_{l=1}^m \lambda_{l,k_l}\Big)
\Big(
\prod_{l=1}^m 
\inf_{z \in R_{\vec{k}}^*}
M\chi_{Q_{l,k_l}}(z)^{\frac{N-s}{m n}}
\Big)
(M\chi_{R_{\vec{k}}^*})^{\frac{n+s+1}{n}}
 \Big\|_{L^p}.
\]
Repeating the argument in estimating for $B$, noting that $\frac{(n+s+1)p}{n}>1$ and $N\gg s$, we obtain
\begin{equation}\label{eq.3A17}
G \lesssim
G_1+ G_2
\lesssim
\prod_{l=1}^m
\Big\|
\sum_{k_l=1}^\infty
\lambda_{l,k_l}
\chi_{Q_{l,k_l}}
\Big\|_{L^{p_l}}.
\end{equation}
Combining \eqref{eq.3A16} and \eqref{eq.3A17} deduces \eqref{eq.PWEST}. This completes the proof of Proposition \ref{LM.Key-31} for the operator $\mathcal{T}_\sigma$ of type \eqref{eq.CalZygOPT}.

\begin{remark}
The techniques in this paper also work for CZ operators of non-convolution types; this 
recovers the results in \cite{HuMe12}. 
\end{remark}

%=============================================================%
\section{The product type}
%=============================================================%
On this whole section, we denote by $\mathcal{T}_\sigma$ the operator defined in  \eqref{eq.CalZygOPT-2} and prove Proposition \ref{LM.Key-31} for this operator.
Now we need to establish some results  analogous to Lemmas \ref{lem:160726-51} and \ref{lm.3A10}.

\subsection{Fundamental estimates for the product type}
Let $a_k$ be $(p_k,\infty)$-atoms supported in $Q_k$ for all $1\le k\le m$.
Here and below
$M^{(r)}$ denotes the power-maximal operator:
$M^{(r)}f(x)=M(|f|^r)(x)^\frac{1}{r}$.
Suppose $Q_1$ is the cube such that $\ell(Q_1) =  \min\{\ell(Q_k) : 1\le k\le m\}$, then we have the following lemmas.
\begin{lemma}
\label{lm.4A00}
For all $x\in Q_1^{**}$, we have
\begin{align}\label{eq.4A01}
M_\phi\circ \mathcal{T}_\sigma(a_1,\ldots,a_m)(x)\chi_{Q_1^{**}}(x)
\lesssim
\prod_{l=1}^m
M\chi_{Q_l}(x)^\frac{n+N+1}{mn}
\left(
1+M^{(m)} \circ T_{\sigma_l}(a_l)(x)
\right).
\end{align}
\end{lemma}
\begin{proof}
Fix $x\in Q_1^{**}$. We need to estimate
\[
\Big|
\int_{{\mathbb R}^n} \phi_t(x-y){\mathcal T}_\sigma(a_1,\ldots,a_m)(y)\,dy
\lesssim
\frac{1}{t^n}
\int_{B(x,t)}
|{\mathcal T}_\sigma(a_1,\ldots,a_m)(y)|\,dy
\]
for each $t\in (0,\infty)$.
The proof of \eqref{eq.4A01} is mainly based on the boundedness of $\mathcal{T}_\sigma$ and the smoothness condition of each Calder\'on-Zygmund kernel in \eqref{eq.CalZygOPT-2}. Instead of considering the whole sum in \eqref{eq.CalZygOPT-2}, 
for notational simplicity, it is convenient to consider one term, i.e.,
\begin{equation}\label{eq.4A02}
 {\mathcal T}_\sigma(f_1,\ldots,f_m) =
T_{\sigma_1}(f_1) \cdots
T_{\sigma_m}(f_m)   
\end{equation}
except when cancellation is used, when the entire sum is needed.
We consider two cases: $t\le \ell(Q_1)$ and $t>\ell(Q_1)$.

\textbf{Case 1: $t\le \ell(Q_1)$}. By   H\"older inequality and \eqref{eq.3D4}, we have
%detail%%%%%%%%%%%%%%%%%%%%%%%%%%%%%%%%%%%%%%%%%%
\if0
\footnote{
The detail is as follows.
We need to decompose two cases; 
$Q_1^{**}\cap Q_l^{**}=\emptyset$ or not.
In fact, 
\begin{align*}
&\prod_{l=1}^m
\Big(
\frac{1}{t^n}
\int_{B(x,t)}
|T_{\sigma_l}a_l(y)|^mdy
\Big)^{\frac1m}\\
=&
\prod_{l:Q_1^{**}\cap Q_l^{**}\neq \emptyset}
\Big(
\frac{1}{t^n}
\int_{B(x,t)}
|T_{\sigma_l}a_l(y)|^mdy
\Big)^{\frac1m}
\prod_{l:Q_1^{**}\cap Q_l^{**}=\emptyset}
\Big(
\frac{1}{t^n}
\int_{B(x,t)}
|T_{\sigma_l}a_l(y)|^mdy
\Big)^{\frac1m}. 
\end{align*}
For the first term, 
we employ the trivial estimate and 
$1=\chi_{Q_1^*}(x)\leq \chi_{Q_l^{**}}(x)$.
For the second term, 
we observe that the assumptions 
$x\in Q_1^*$ and $t\leq \ell(Q_1)$
imply $B(x,t)\subset Q_1^{**}$ and hence 
$B(x,t)\cap Q_l^{**}=\emptyset$.
This observation allows us to use \eqref{eq.3D4}.
As a result, 
\[
\prod_{l=1}^m
\Big(
\frac{1}{t^n}
\int_{B(x,t)}
|T_{\sigma_l}a_l(y)|^mdy
\Big)^{\frac1m}
\lesssim
\prod_{l:Q_1^{**}\cap Q_l^{**}\neq \emptyset}
\chi_{Q_l^{**}}(x)
M|T_{\sigma_l}(a_l)|^m(x)^{\frac{1}{m}}
\prod_{l:Q_1^{**}\cap Q_l^{**}=\emptyset}
M\chi_{Q_l}(x)^\frac{n+N+1}{mn}.
\] 
}
\fi
%%%%%%%%%%%%%%%%%%%%%%%%%%%%%%%%%%%%%%
\begin{align*}
\frac{1}{t^n}
\int_{B(x,t)}
|{\mathcal T}_\sigma(a_1,\ldots,a_m)(y)|\,dy
\lesssim&
\prod_{l=1}^m
\Big(
\frac{1}{t^n}
\int_{B(x,t)}
|T_{\sigma_l}a_l(y)|^mdy
\Big)^{\frac1m}. 
\end{align*}
Now, we decompose the above product depending on two sub-cases; 
$B(t,x)\cap Q_l^{**}=\emptyset$ or not.
Then
\begin{align*}
&\prod_{l=1}^m
\Big(
\frac{1}{t^n}
\int_{B(x,t)}
|T_{\sigma_l}a_l(y)|^mdy
\Big)^{\frac1m}\\
=&
\prod_{l:B(t,x)\cap Q_l^{**}=\emptyset}
\Big(
\frac{1}{t^n}
\int_{B(x,t)}
|T_{\sigma_l}a_l(y)|^mdy
\Big)^{\frac1m}
\prod_{l:B(t,x)\cap Q_l^{**}\neq \emptyset}
\Big(
\frac{1}{t^n}
\int_{B(x,t)}
|T_{\sigma_l}a_l(y)|^mdy
\Big)^{\frac1m}. 
\end{align*}
For the first sub-case, 
we employ 
\eqref{eq.3D4}. For the second sub-case,
we observe that the assumption
$t\leq \ell(Q_1)\le \ell(Q_l)$
imply $B(x,t)\subset 3Q_l^{**}$.
As a result, 
\begin{align*}
\lefteqn{
\prod_{l=1}^m
\Big(
\frac{1}{t^n}
\int_{B(x,t)}
|T_{\sigma_l}a_l(y)|^mdy
\Big)^{\frac1m}
}\\
&\lesssim
\prod_{l:B(x,t)\cap Q_l^{**}\neq \emptyset}
\chi_{3Q_l^{**}}(x)
M^{(m)} \circ T_{\sigma_l}(a_l)(x)
\prod_{l:B(x,t)\cap Q_l^{**}=\emptyset}
M\chi_{Q_l}(x)^\frac{n+N+1}{mn}.
\end{align*} 
Thus
\begin{equation}
\label{eq.4A03}
\Big|
\int_{{\mathbb R}^n} \phi_t(x-y){\mathcal T}_\sigma(a_1,\ldots,a_m)(y)\,dy
\Big|
\lesssim
\prod_{l=1}^m
M\chi_{Q_l}(x)^\frac{n+N+1}{mn}
\left(
1+M^{(m)} \circ T_{\sigma_l}(a_l)(x)
\right).
\end{equation}

\textbf{Case 2: $t>\ell(Q_1)$}. Now we can estimate
\begin{align*}
\frac{1}{t^n}
\int_{B(x,t)}
|{\mathcal T}_\sigma(a_1,\ldots,a_m)(y)|\,dy
\lesssim&
\frac{1}{|Q_1^*|}
\int_{\mathbb{R}^n}
|{\mathcal T}_\sigma(a_1,\ldots,a_m)(y)|\,dy\\
=&
\frac{1}{|Q_1^*|}
\int_{Q_1^*}
|{\mathcal T}_\sigma(a_1,\ldots,a_m)(y)|\,dy\\
&+
\frac{1}{|Q_1^*|}
\int_{\mathbb{R}^n\setminus Q_1^*}
|{\mathcal T}_\sigma(a_1,\ldots,a_m)(y)|\,dy.
\end{align*}
By the H\"older inequality and \eqref{eq.3D4}, 
the similar technique to \eqref{eq.4A03} 
yields 
\begin{align}
\notag
\frac{1}{|Q_1^*|}
\int_{Q_1^*}
|{\mathcal T}_\sigma(a_1,\ldots,a_m)(y)|\,dy
\lesssim&
\prod_{l=1}^m
\Big(
\frac{1}{|Q_1^*|}
\int_{Q_1^*}
|T_{\sigma_l}a_l(y)|^m\,dy
\Big)^{\frac1m}
\\
\notag
\lesssim&
\prod_{l=1}^m
\Big(
\inf_{z \in Q_{1}^*}
M\chi_{Q_{l}^{**}}(z)^{\frac{n+N+1}{mn}}
+
\inf_{z \in Q_{1}^*}
M^{(m)} \circ T_{\sigma_l}(a_l)(z)
\chi_{3Q_l^{**}}(z)
\Big)
\\
\label{eq.4A05}
\lesssim&
\prod_{l=1}^m
M\chi_{Q_l}(x)^\frac{n+N+1}{mn}
\left(
1+M^{(m)} \circ T_{\sigma_l}(a_l)(x)
\right), 
\end{align}
since $x\in Q_1^*$.
For the second term, using the decay of $T_{\sigma_1}a_1(y)$  when $y\notin Q_1^*$ as in Lemma \ref{lm.3A1},
we obtain
$$%\begin{align*}
%&
\frac{1}{|Q_1^*|}
\int_{{\mathbb R}^n \setminus Q_1^*}
|{\mathcal T}_\sigma(a_1,\ldots,a_m)(y)|\,dy
%\\
%&
\lesssim
\frac{1}{|Q_1^*|}
\int_{{\mathbb R}^n \setminus Q_1^*}
\frac{\ell(Q_1)^{n+N+1}}{|y-c_1|^{n+N+1}}
\prod_{l=2}^m |T_{\sigma_l}a_l(y)|\,dy.
$$%\end{align*}
We decompose ${\mathbb R}^n \setminus Q_1^*$
into dyadic annuli and estimate
\begin{align*}
&
\frac{1}{|Q_1^*|}
\int_{{\mathbb R}^n \setminus Q_1^*}
|{\mathcal T}_\sigma(a_1,\ldots,a_m)(y)|\,dy
\\
&\lesssim
\sum_{j=1}^\infty
2^{j(-N-1)}
\frac{1}{|2^jQ_1^*|}
\int_{2^{j}Q_1^*}
\chi_{2^{j}Q_1^*}(y)
\prod_{l=2}^m |T_{\sigma_l}a_l(y)|\,dy\\
&\lesssim
\sum_{j=1}^\infty
2^{j(-N-1)}
\prod_{l=2}^m
\Big(
\frac{1}{|2^jQ_1^*|}\int_{2^jQ_1^*}|T_{\sigma_l}a_l(y)|^m\,dy
\Big)^{\frac1m}\\
&\lesssim
\sum_{j=1}^\infty
2^{j(-N-1)}
\prod_{l=2}^m
\Big(
\inf_{z\in 2^jQ_1^*}(M\chi_{2^jQ_l^{**}})(z)^{\frac{n+N+1}{mn}}
+
\inf_{z\in 2^jQ_1^*}
M^{(m)} \circ T_{\sigma_l}(a_l)(z)
\chi_{2^{j+1}Q_l^{**}}(z)
\Big), %\\
%&\lesssim
%\prod_{l=2}^m
%\sum_{j=1}^\infty
%\Big(
%2^{-\frac{j(2N+2)}{2m-1}}
%\inf_{z\in 2^jQ_1^*}(M\chi_{2^jQ_l^{**}})(z)^{\frac{n+N+1}{n}}
%+
%2^{-\frac{j(2N+2)}{2m-1}}
%\inf_{z\in Q_1^*}
%M|T_{\sigma_l}a_l|^m(z)^{\frac1m}
%\chi_{2^{j+1}Q_l^{**}}(z)
%\Big),
\end{align*}
where we used \eqref{eq.3D4} in the last inequality.

Since $M\chi_{2^jQ}\lesssim 2^{jn}M\chi_Q$,
\[
%(M\chi_{2^jQ_l^{**}})^{\frac{n+N+1}{n}}
\chi_{2^{j+1}Q_l^{**}}(x)
\le 
(M\chi_{2^jQ_l^{**}})^{\frac{n+N+1}{mn}}
\lesssim
2^{\frac{j(n+N+1)}{m}}
M\chi_{Q_l}^{\frac{n+N+1}{mn}}.
\]
Insert this inequality into the previous estimate to obtain
\begin{align}
\notag
&\frac{1}{|Q_1^*|}
\int_{{\mathbb R}^n \setminus Q_1^*}
|{\mathcal T}_\sigma(a_1,\ldots,a_m)(y)|\,dy
\\
\notag
&\lesssim
\sum_{j=1}^\infty
2^{-j(\frac{n+N+1}{m}-n)}
\prod_{l=1}^m
\Big(
M\chi_{Q_l}(x)^{\frac{n+N+1}{mn}}
+
M^{(m)} \circ T_{\sigma_l}(a_l)(x)
M\chi_{Q_l}(x)^\frac{n+N+1}{mn}
\Big)
\\
\label{eq.4E00}
&\lesssim
\prod_{l=1}^m
M\chi_{Q_l}(x)^\frac{n+N+1}{mn}
\left(
1+M^{(m)} \circ T_{\sigma_l}(a_l)(x)
\right), 
\end{align}
since $N \gg n$.
Combining \eqref{eq.4A03}--\eqref{eq.4E00} together completes the proof of \eqref{eq.4A01}.
\end{proof}

\begin{lemma}\label{lm.4F01}
Assume
$x \notin Q_1^{**}$ and $c_1 \notin B(x,100n^2t)$.
Then %if we choose $B^1_{l,0}=b_{l,0}$ satisfying $(\ref{eq:160726-1})$ suitably,
%then
we have 
\begin{align*}
\frac{1}{t^n}
\int_{B(x,t)}|{\mathcal T}_\sigma(a_1,\ldots,a_m)(y)|\,dy
\lesssim 
\prod_{l=1}^m 
M\chi_{Q_l}(x)^\frac{n+N+1}{mn}
\left(
1+M^{(m)} \circ T_{\sigma_l}(a_l)(x)
\right).
\end{align*}
\end{lemma}

\begin{proof}
Fix any $x\notin Q_1^{**}$ and $t>0$ such that 
$c_1\notin B(x,100n^2t)$. 
We denote
\begin{equation}
\label{lm.4F02}
J=
\{
2\le l\le m: 
x\notin Q_l^{**}
\},\ 
J_0 = \{l\in J : B(x,2t)\cap Q_l^*=\emptyset\},\
J_1 = J\setminus J_0.
\end{equation}
Similar to the previous lemma, it is enough to consider the reduced form \eqref{eq.4A02} of $\mathcal{T}_\sigma$.
From the H\"{o}lder inequality, we have 
\begin{align*}
&\frac{1}{t^n}
\int_{B(x,t)}
|\mathcal{T}_\sigma(a_1,\ldots,a_m)(y)|dy\\
&\lesssim
\|
T_{\sigma_1}a_1\chi_{B(x,t)}
\|_{L^\infty}
\prod_{l\in J_0}
\|T_{\sigma_l}a_l\chi_{B(x,t)}\|_{L^{\infty}}\\
&\quad \times
\prod_{l\in J_1}
\left(
\frac{1}{|B(x,t)|}
\int_{B(x,t)}|T_{\sigma_l}a_l(y)|^mdy
\right)^{\frac{1}{m}}
%&\quad \times 
\prod_{l\notin J}
\left(
\frac{1}{|B(x,t)|}
\int_{B(x,t)}|T_{\sigma_l}a_l(y)|^mdy
\right)^{\frac{1}{m}}\\
&=:
{\rm I}\times {\rm II}\times {\rm III}\times {\rm IV}.
\end{align*}
For $\rm I$, we notice 
$Q_1^*\cap B(x,2t)=\emptyset$ since 
we have 
$x\notin Q_1^{**}$ and 
$c_1\notin B(x,100n^2t)$. 
%\footnote{
%In fact, 
%if $Q_1^*\cap B(x,t)\neq \emptyset$, 
%$x\notin Q_1^**$ implies 
%$t\gtrsim \ell(Q_1)$. 
%Then from $B(x,t)\cap Q_1^*\neq \emptyset$ again, 
%we see 
%$
%Q_1\subset 
%B(x,100n^2t). 
%$
%But, this is a contradiction to 
%$c_1\notin B(x,100n^nt)$.
%}
So, 
we have only to use the decay estimate for $T_{\sigma_1}a_1$ to get 
\[
{\rm I}=
\|T_{\sigma_1}a_1\chi_{B(x,t)}\|_{L^\infty}
\lesssim 
\left(
\frac{\ell(Q_1)}{
|x-c_1|+\ell(Q_1)
}
\right)^{n+N+1}.
\]

For all 
$l\in J_1$,
since $B(x,2t)\cap Q_l^*\neq\emptyset$, $t\gtrsim \ell(Q_l)$; 
and hence, $Q_l^*\subset B(x,100n^2t)$. Therefore,
\begin{equation}
\label{eq.4E01}
\left(
\frac{1}{|B(x,t)|}
\int_{B(x,t)}|T_{\sigma_l}a_l(y)|^mdy
\right)^{\frac{1}{m}}
\lesssim 
\left(
\frac{|Q_l|}{|B(x,t)|}
\right)^{\frac{1}{m}}
\lesssim 1. 
\end{equation}
for all $l\in J_1$.
Now combining the above inequality with the estimates 
for $\rm I$ yields
\begin{equation}
\label{eq.4B01}
{\rm I}\times {\rm III}
\lesssim 
\left(
\frac{\ell(Q_1)}{
|x-c_1|+\ell(Q_1)
}
\right)^{\frac{n+N+1}{m}}
\prod_{l\in J_1}
\left(
\frac{\ell(Q_1)}{|x-c_1|+\ell(Q_1)}
\right)^{\frac{n+N+1}{m}}.
\end{equation}
As showed about $Q_l^*\subset B(x,100n^2t)$ for all $l\in J_1$. This implies $|x-c_l|\lesssim t$. Furthermore, 
$c_1\notin B(x,100n^2t)$ means 
$t\lesssim |x-c_1|$ which yields $|x-c_l|\lesssim t\lesssim |x-c_1|$.

Recalling $\ell(Q_1)\leq \ell(Q_l)$, 
we see that
\[
\frac{\ell(Q_1)}{
|x-c_1|+\ell(Q_1)
}
\lesssim 
\frac{\ell(Q_l)}{
|x-c_l|+\ell(Q_l)
}.
\]

%%%%%%%%%%%%%%%
%%%%%%%%%%%%%%%

From \eqref{eq.4B01}, we obtain 
\begin{align}
%\nonumber
{\rm I}\times {\rm III}
\lesssim 
M\chi_{Q_1^{**}}(x)^{\frac{n+N+1}{m n}}
\label{eq:161003-11}
\prod_{l\in J_1}
M\chi_{Q_l}(x)^\frac{n+N+1}{m n}.
\end{align}

Now, we turn to the estimate 
for ${\rm II}$ and ${\rm IV}$. 
For ${\rm II}$, we have only to employ 
the moment condition of $a_l$
to get 
\begin{align}
\label{eq:161003-12}
{\rm II}
&=
\prod_{l\in J_0}
\|T_{\sigma_l}a_l \cdot \chi_{B(x,t)}\|_{L^\infty}
\lesssim
%\prod_{l\in J_x; B(x,t)\cap Q_l^*=\emptyset}
%\left(
%\frac{\ell(Q_l)}{|x-c_l|+\ell(Q_l)}
%\right)^{n+N+1}
%|Q_l|^{-\frac{1}{p_l}}\\
%&\lesssim
\prod_{l\in J_0}
M\chi_{Q_l}(x)^{\frac{n+N+1}{n}}.
\end{align}
For ${\rm IV}$, since $x\in Q_l^{**}$, we can estimate 
\begin{align}
\label{eq:161003-13}
{\rm IV}
&\lesssim 
\prod_{l\notin J}
M^{(m)} \circ T_{\sigma_l}(a_l)(x)
\chi_{Q_l^{**}}(x)
\end{align}
%Modifying large numbers $N,M$ suitably again 
Putting (\ref{eq:161003-11})--(\ref{eq:161003-13}) together, 
we conclude the proof of Lemma \ref{lm.4F01}. 
\end{proof}

\begin{lemma}\label{lem:160722-3}
Assume
$x \notin Q_1^{**}$ and $c_1 \in B(x,100n^2t)$.
Then we have 
\begin{align}
\notag
&
\frac{\ell(Q_1)^{s+1}}{t^{n+s+1}}
\int_{Q_1^{*}}|{\mathcal T}_\sigma(a_1,\ldots,a_m)(y)|\,dy\\
\label{eq:160726-123}
&
\lesssim M\chi_{Q_1}(x)^{\frac{n+s+1}{n}}
\prod_{l=1}^m
\inf_{z\in Q_1^*}
M\chi_{Q_l}(z)^\frac{n+N+1}{mn}
\left(
1+M^{(m)} \circ T_{\sigma_l}(a_l)(z)
\right).
\end{align}
\end{lemma}
\begin{proof}
It is enough to restrict $\mathcal{T}_\sigma$ to the form  \eqref{eq.4A02}. By the H\"older inequality we have
\begin{align*}
&
\frac{\ell(Q_1)^{s+1}}{t^{n+s+1}}
\int_{Q_1^*}|{\mathcal T}_\sigma(a_1,\ldots,a_m)(y)|\,dy\\
\le&
\frac{\ell(Q_1)^{n+s+1}}{t^{n+s+1}}
\prod_{l=1}^m
\Big(
\frac{1}{|Q_1^*|}
\int_{Q_1^*}|T_{\sigma_l}a_l(y)|^mdy
\Big)^{\frac1m}\\
\lesssim&
\frac{\ell(Q_1)^{n+s+1}}{t^{n+s+1}}
\prod_{l=1}^m
\Big(
\inf_{z\in Q_1^*}
M\chi_{Q_l}(z)^{\frac{n+N+1}{n}}
+
\inf_{z\in Q_1^*}
M^{(m)} \circ T_{\sigma_l}(a_l)(z)
\chi_{2Q_l^{**}}(z)
\Big),
\end{align*}
where the last inequality is deduced from \eqref{eq.3D4}.

Since $x\notin Q_1^{**}$ and $c_1\in B(x,100n^2t)$, 
$Q_1 \subset B(x,10000n^3t)$ which implies 
$
\ell(Q_1)/t\lesssim M\chi_{Q_1}(x)^\frac{1}{n}
$. %if $Q_1^* \cap Q_l^*=\emptyset$.
As a result,
\begin{align*}
&
\frac{\ell(Q_1)^{s+1}}{t^{n+s+1}}
\int_{Q_1^*}|{\mathcal T}_\sigma(a_1,\ldots,a_m)(y)|\,dy\\
&\lesssim
M\chi_{Q_1}(x)^\frac{n+s+1}{n}
\prod_{l=1}^m
\inf_{z \in Q_1^*}
M\chi_{Q_l}(z)^\frac{n+N+1}{mn}
\left(
1+M^{(m)} \circ T_{\sigma_l}(a_l)(z)
\right).
\end{align*}
This proves
(\ref{eq:160726-123}).
\end{proof}

\begin{lemma}\label{lem:160722-4}
Assume
$x \notin Q_1^{**}$ and $c_1 \in B(x,100n^2t)$.
Then %if we choose $B^4_{l,j}=b_{l,j}$ satisfying $(\ref{eq:160726-1})$ suitably,
%then 
we have 
\begin{align*}
&
\frac{1}{t^{n+s+1}}
\int_{{\mathbb R}^n \setminus Q_1^*}
|y-c_1|^{s+1}|{\mathcal T}_\sigma(a_1,\ldots,a_m)(y)|\,dy\\
&
\lesssim 
M\chi_{Q_1}(x)^{\frac{n+s+1}{n}}
\prod_{l=1}^m
\inf_{z\in Q_1^*}M\chi_{Q_l}(z)^{\frac{n+N+1}{mn}}
\left(
1+M^{(m)} \circ T_{\sigma_l}(a_l)(z)
\right).
\end{align*}
\end{lemma}
\begin{proof}
Using the decay of $T_{\sigma_1}a_1(y)$ when $y\notin Q_1^*$,
we obtain
\begin{align*}
\lefteqn{
\frac{1}{t^{n+s+1}}
\int_{{\mathbb R}^n \setminus Q_1^*}|y-c_1|^{s+1}|{\mathcal T}_\sigma(a_1,\ldots,a_m)(y)|\,dy
}\\
&\lesssim
\frac{1}{t^{n+s+1}}
\int_{{\mathbb R}^n \setminus Q_1^*}
|y-c_1|^{s+1}
\frac{\ell(Q_1)^{n+N+1}}{|y-c_1|^{n+N+1}}
\prod_{l=2}^m |T_{\sigma_l}a_l(y)|\,dy.
\end{align*}
By dyadic decomposition of ${\mathbb R}^n \setminus Q_1^*$ as in the proof of Lemma \ref{lm.4A00}, we can estimate
\begin{align*}
&\frac{1}{t^{n+s+1}}
\int_{{\mathbb R}^n \setminus Q_1^*}|y-c_1|^{s+1}|{\mathcal T}_\sigma(a_1,\ldots,a_m)(y)|\,dy
\\
&\lesssim
\frac{\ell(Q_1)^{s+1}}{t^{n+s+1}}
\sum_{j=1}^\infty
2^{j(s-N-n)}
\int_{2^{j}Q_1^*}
\chi_{2^{j}Q_1^*}(y)
\prod_{l=2}^m |T_{\sigma_l}a_l(y)|\,dy\\
&\lesssim
\frac{\ell(Q_1)^{n+s+1}}{t^{n+s+1}}
\sum_{j=1}^\infty
2^{j(s-N)}
\prod_{l=2}^m
\Big(
\frac{1}{|2^jQ_1^*|}\int_{2^jQ_1^*}|T_{\sigma_l}a_l(y)|^m\,dy
\Big)^{\frac1m}\\
&\lesssim
\frac{\ell(Q_1)^{n+s+1}}{t^{n+s+1}}
\sum_{j=1}^\infty
2^{j(s-N)}
\prod_{l=2}^m
\Big(
\inf_{z\in 2^jQ_1^*}M\chi_{2^jQ_l^{**}}(z)^{\frac{n+N+1}{mn}}
+
\inf_{z\in 2^jQ_1^*}
M^{(m)} \circ T_{\sigma_l}(a_l)(z)
\chi_{2^{j+1}Q_l^{**}}(z)
\Big),
\end{align*}
where we used \eqref{eq.3D4} in the last inequality.

We now repeat the argument in establishing \eqref{eq.4E00} to obtain
\begin{align*}
&\frac{1}{t^{n+s+1}}
\int_{{\mathbb R}^n \setminus Q_1^*}|y-c_1|^{s+1}|{\mathcal T}_\sigma(a_1,\ldots,a_m)(y)|\,dy
\\
&\lesssim
\frac{\ell(Q_1)^{n+s+1}}{t^{n+s+1}}
\prod_{l=1}^m
\inf_{z\in Q_1^*}M\chi_{Q_l}(z)^{\frac{n+N+1}{mn}}
\left(
1+M^{(m)} \circ T_{\sigma_l}(a_l)(z)
\right).
\end{align*}

Moreover, the assumption 
$x\notin Q_1^{**}$ and 
$c_1\in B(x,100n^2t)$ implies 
$
\frac{\ell(Q_1)}{t}\lesssim M\chi_{Q_1}(x)^{\frac1n}. 
$
Therefore,
\begin{align*}
&
\frac{1}{t^{n+s+1}}
\int_{{\mathbb R}^n \setminus Q_1^*}|y-c_1|^{s+1}|{\mathcal T}_\sigma(a_1,\ldots,a_m)(y)|\,dy\\
&\lesssim
M\chi_{Q_1}(x)^{\frac{n+s+1}{n}}
\prod_{l=1}^m
\inf_{z\in Q_1^*}M\chi_{Q_l}(z)^{\frac{n+N+1}{mn}}
\left(
1+M^{(m)} \circ T_{\sigma_l}(a_l)(z)
\right).
\end{align*}
This proves
Lemma \ref{lem:160722-4}.
\end{proof}

\begin{lemma}
\label{lm.4B02}
For all $x\in \R^n$, we have
\begin{align*}
M_\phi\circ \mathcal{T}_\sigma(a_1,\ldots,a_m)(x)
\lesssim&
\prod_{l=1}^m
M\chi_{Q_l}(x)^{\frac{n+N+1}{mn}}
\left(
1+M^{(m)} \circ T_{\sigma_l}(a_l)(x)
\right)\\
&+
M\chi_{Q_1}(x)^{\frac{n+s+1}{n}}
\prod_{l=1}^m
\inf_{z\in Q_1^*}M\chi_{Q_l}(z)^{\frac{n+N+1}{mn}}
\left(
1+M^{(m)} \circ T_{\sigma_l}(a_l)(z)
\right).
\end{align*}
\end{lemma}
\begin{proof}
If $x\in Q_1^{**}$, 
the desired estimate is a consequence of Lemma \ref{lm.4A00}.  
Fix $x\notin Q_1^{**}$.
To estimate 
$M_{\phi}\circ \mathcal{T}_{\sigma}(a_1,\ldots,a_m)(x)$, 
we need to examine
\[
\Big|
\int_{{\mathbb R}^n} \phi_t(x-y){\mathcal T}_\sigma(a_1,\ldots,a_m)(y)\,dy
\Big|
\]
for each $t\in (0,\infty)$. If $c_1 \notin B(x,100n^2t)$, then we make use of Lemma \ref{lm.4F01}; otherwise, when $c_1 \in B(x,100n^2t)$ we recall \eqref{eq.3C02} and then apply Lemma \ref{lem:160722-3} and \ref{lem:160722-4} to obtain the required estimate in Lemma \ref{lm.4B02}. This completes the proof of the lemma.
\end{proof}

\subsection{The proof of Proposition \ref{LM.Key-31} for the product type}
To process the proof of  \eqref{eq.PWEST}, we set
\begin{equation*}
A= \Big\|\sum_{k_1,\ldots,k_m=1}^\infty
\Big(\prod_{l=1}^m \lambda_{l,k_l}\Big)
 M_\phi \circ {\mathcal T}_\sigma(a_{1,k_1},\ldots,a_{m,k_m})\Big\|_{L^p}.
\end{equation*}
For each $\vec{k}=(k_1,\ldots,k_m)$, we recall $R_{\vec{k}}$, the smallest-length cube among $Q_{1,k_1},\ldots,Q_{m,k_m}$.

In view of Lemma \ref{lm.4B02}, 
we have 
\begin{equation}\label{161031-1}
A\lesssim 
B:=
\left\|
\sum_{k_1,\ldots, k_m=1}^\infty
\prod_{l=1}^m
\lambda_{l,k_l}
\left(M\chi_{Q_{l,k_l}}\right)^{\frac{n+N+1}{mn}}
\left(
1+M^{(m)} \circ T_{\sigma_l}(a_{l,k_l})
\right)
\right\|_{L^p}. 
\end{equation}
In fact, our assumption imposing on $s$ means 
$(n+s+1)p/n>1$ and hence we may employ the boundedness of $M$ 
to obtain 
\begin{align*}
A&\lesssim 
B+
\left\|
\sum_{k_1,\ldots,k_m=1}^\infty
\left(M\chi_{R_{\vec{k}}^*}\right)^{\frac{n+s+1}{n}}
\prod_{l=1}^m
\lambda_{l,k_l}
\inf_{z\in R_{\vec{k}}^*}
\bigg(M\chi_{Q_{l,k_l}}\bigg) ^{\frac{n+N+1}{mn}}
\left(
1+M^{(m)} \circ T_{\sigma_l}(a_{l,k_l})
\right)
\right\|_{L^p}\\
&\lesssim
B.
\end{align*}

So, our task is to estimate $B$. 
Here, 
we prepare the following lemma. 
\begin{lemma}\label{lm-161031-1}
Let $p \in(0,\infty)$ and $\alpha>\max{(1,p^{-1})}$.
Assume that $q \in (p,\infty] \cap [1,\infty]$.
Suppose that we are given
a sequence of cubes $\{Q_k\}_{k=1}^\infty$
and a sequence of non-negative $L^q$-functions $\{F_k\}_{k=1}^\infty$.
Then
\[
\Big\|
\sum_{k=1}^\infty (M\chi_{Q_k})^\alpha F_k
\Big\|_{L^p}
\lesssim
\Big\|
\sum_{k=1}^\infty \chi_{Q_k}
M^{(q)}F_k
\Big\|_{L^p}.
\]
\end{lemma}

\begin{proof}
By Lemma \ref{lm.2A00} and the fact that
\(
M\chi_{Q}
\lesssim 
\chi_{Q}+
\sum_{j=1}^{\infty}
2^{-jn}\chi_{2^{j}Q\setminus 2^{j-1}Q},
\)
we have
\begin{align*}
\left\|
\sum_{k=1}^\infty (M\chi_{Q_k})^\alpha F_k
\right\|_{L^p}
&\lesssim
\left\|
\sum_{j=0}^\infty
\sum_{k=1}^\infty
2^{-\alpha jn}
\chi_{2^jQ_k}F_k
\right\|_{L^p}\\
&\lesssim
\left\|
\sum_{j=0}^\infty
\sum_{k=1}^\infty
2^{-\alpha jn}
\chi_{2^jQ_k}
\left(
\frac{1}{|2^jQ_k|}
\int_{2^jQ_k}F_k(y)^qdy
\right)^\frac{1}{q}
\right\|_{L^p}.
\end{align*}
Choose $\alpha>\beta>\max(1,\frac1p)$ and observe the trivial estimate 
\[
\chi_{2^jQ_k}
\lesssim 
\left(
2^{jn}M\chi_{Q_k}
\right)^{\beta}.
\] 
Now, Lemma \ref{lm.2B00} gives
\begin{align*}
\left\|
\sum_{k=1}^\infty (M\chi_{Q_k})^\alpha F_k
\right\|_{L^p}
\lesssim&
\left\|
\sum_{j=0}^\infty
\sum_{k=1}^\infty
\chi_{Q_k}
\left(
\frac{2^{(\beta-\alpha)jqn}}{|2^jQ_k|}
\int_{2^jQ_k}F_k(y)^qdy
\right)^\frac{1}{q}
\right\|_{L^p}
\lesssim&
\Big\|
\sum_{k=1}^\infty \chi_{Q_k}
M^{(q)}F_k
\Big\|_{L^p},
\end{align*}
which yields the desired estimate.
\end{proof}
Lemma \ref{lm-161031-1} can be regard as a substitution 
of Lemma \ref{lm.2A00}.

Before applying Lemma \ref{lm-161031-1} to $B$, 
we observe 
\begin{align*}
B
\leq
\prod_{l=1}^m
\left\|
\sum_{k_l=1}^\infty
\lambda_{l,k_l}
\left(M\chi_{Q_{l,k_l}}\right)^{\frac{n+N+1}{mn}}
\left(
1+M^{(m)} \circ T_{\sigma_l}(a_{l,k_l})
\right)
\right\|_{L^{p_l}}.
\end{align*}
Then applying Fefferman-Stein's vector-valued inequality 
and Lemma \ref{lm-161031-1}, 
\begin{align*}
&\left\|
\sum_{k_l=1}^\infty
\lambda_{l,k_l}
\left(M\chi_{Q_{l,k_l}}\right)^{\frac{n+N+1}{mn}}
\left(
1+  M^{(m)} \circ T_{\sigma_l}(a_{l,k_l})
\right)
\right\|_{L^{p_l}}\\
&\lesssim 
\left\|
\sum_{k_l=1}^\infty
\lambda_{l,k_l}
  M \big( \chi_{Q_{l,k_l}^{**}}\big)^{\frac{n+N+1}{mn}}
\right\|_{L^{p_l}}
+
\left\|
\sum_{k_l=1}^\infty
\lambda_{l,k_l}
\left(M\chi_{Q_{l,k_l}}\right)^{\frac{n+N+1}{mn}}
M^{(m)} \circ T_{\sigma_l}(a_{l,k_l})
\right\|_{L^{p_l}}\\
&\lesssim
\left\|
\sum_{k_l=1}^\infty
\lambda_{l,k_l}
\chi_{Q_{l,k_l}}
\right\|_{L^{p_l}}
+
\left\|
\sum_{k_l=1}^\infty
\lambda_{l,k_l}
\chi_{Q_{l,k_l}^{**}}
M\circ M^{(m)} \circ T_{\sigma_l}(a_{l,k_l})
\right\|_{L^{p_l}}. 
\end{align*}
For the second term, 
we choose $q\in(m,\infty)$ and
employ Lemma \ref{lm.2A00}, 
and the boundedness of $M$ and $T_{\sigma_l}$ 
to have 
\begin{align*}
&\left\|
\sum_{k_l=1}^\infty
\lambda_{l,k_l}
\chi_{Q_{l,k_l}^{**}}
M\circ M^{(m)} \circ T_{\sigma_l}(a_{l,k_l})
\right\|_{L^{p_l}}\\
&\lesssim
\left\|
\sum_{k_l=1}^\infty
\lambda_{l,k_l}
\frac{\chi_{Q_{l,k_l}^{**}}}{|Q_{l,k_l}|^\frac{1}{q}}
\left\|M\circ M^{(m)} \circ T_{\sigma_l}(a_{l,k_l})\right\|_{L^q}^q
\right\|_{L^{p_l}}
\lesssim 
\left\|
\sum_{k_l=1}^\infty
\lambda_{l,k_l}
\chi_{Q_{l,k_l}}
\right\|_{L^{p_l}}.
\end{align*}
As a result, 
\[
A\lesssim B\lesssim 
\prod_{l=1}^m
\left\|
\sum_{k_l=1}^\infty
\lambda_{l,k_l}
\chi_{Q_{l,k_l}}
\right\|_{L^{p_l}}, 
\]
which completes the proof of Proposition 
\ref{LM.Key-31}.

%=============================================================%
\section{The mixed type}
%=============================================================%

In this section, we   prove Proposition \ref{LM.Key-31} for operators of type \eqref{eq.CalZygOPT-3}. The main techniques to deal with the operator $\mathcal{T}_\sigma$ of mixed type are combinations of two previous types. We now establish some necessary estimates for $\mathcal{T}_\sigma$.
For the mixed type, we need the following lemma 
which can be shown by   a way similar to that in Lemma \ref{lm.3A1}.\footnote{
The detailed proof is as follows.
Fix any $y\in 2^{j+1}Q_1^*\setminus 2^jQ_1^*$. 
Let us use a notation $K^1(y,y_1,\ldots, y_m)$ 
as in the proof of Lemma \ref{lm.3A1}. 
Then for any $y_l\in Q_l$, $l=1,\ldots,m$, 
we have 
\begin{align*}
|
K^1(y,y_1,\ldots,y_m)
|
&\lesssim 
\left(
\frac{
\ell(Q_1)
}
{
|y-y_1|
+
\sum_{l\in \Lambda_j}
|y-y_l|
+
\sum_{l\geq2}
|y-y_l|
}
\right)^{n+N+1}\\
&\lesssim
\left(
\frac{
\ell(Q_1)
}
{
|y-c_1|
+
\sum_{l\in \Lambda_j}
|y-c_l|
+
\sum_{l\geq2}
|y-y_l|
}
\right)^{n+N+1}.
\end{align*}
In fact, 
if $l\in \Lambda_j$, 
$2^jQ_1^{**}\cap 2^jQ_l^{**}=\emptyset$ 
and hence, 
$y\in 2^{j+1}Q_1^*$ means 
$|y-y_l|\sim |y-c_l|$ for all $y_l\in Q_l$ for such $l$. 
Of course, 
$|y-y_1|\sim |y-c_1|$ is clear since $y\notin 2^jQ_1^*$.
Using this kernel estimate, 
we may prove the desired estimate.
}

\begin{lemma}\label{lm-161112-1}
Let $\sigma$ be a Coifman-Meyer multiplier, 
$a_l$ be $(p_l,\infty)$-atoms supported on $Q_l$ for 
$1\leq l\leq m$. 
Assume $\ell(Q_1)=\min{\{\ell(Q_l):l=1,\ldots, m\}}$ 
and write 
$\Lambda_j=\{l=1,\ldots,m: 2^jQ_1^{**}\cap 2^jQ_l^{**}=\emptyset\}$. 
Then for any 
$y\in 2^{j+1}Q_1^*\setminus 2^jQ_1^*$ we have 
\[
|T_\sigma(a_1,\ldots, a_m)(y)|
\lesssim 
\left(
\frac{\ell(Q_1)}
{
|y-c_1|
+
\sum_{l\in \Lambda_j}
|y-c_l|}
\right)^{n+N+1}.
\] 
\end{lemma}

\subsection{Fundamental estimates for the mixed type}
Let $a_k$ be $(p_k,\infty)$-atoms supported in $Q_k$ for all $1\le k\le m$.
Suppose $Q_1$ is the cube such that $\ell(Q_1) =  \min\{\ell(Q_k) : 1\le k\le m\}$.
For each $1\le g\le G$,
let $Q_{l(g)}$ be the smallest cube
among $\{Q_l\}_{l \in I_g}$ and let $m_g = |I_g|$ be the cardinality of $I_g$.
Then we have the following analogues to Lemmas \ref{lm.4A00}--\ref{lm.4B02}.
We write $m_g=\sharp I_g$ for each $g$.

\begin{lemma}
\label{lm.5A00}
For all $x\in Q_1^{**}$, we have
\begin{align}
\notag
&M_\phi\circ \mathcal{T}_\sigma(a_1,\ldots,a_m)(x)\chi_{Q_1^{**}}(x)\\
\label{eq.5A01}
&\lesssim
\prod_{g=1}^G
\left(
M\chi_{Q_{l(g)}}(x)^\frac{(n+N+1)m_g}{nm}
M^{(G)} \circ T_{\sigma_{I_g}}(\{a_{l,k_l}\}_{l\in I_g})(x)
+
\prod_{l\in I_g}
M\chi_{Q_l}(x)^\frac{n+N+1}{mn}
\right).
\end{align}
\end{lemma}
\begin{proof}
Fix $x\in Q_1^{**}$. We need to estimate
\[
\Big|
\int_{{\mathbb R}^n} \phi_t(x-y){\mathcal T}_\sigma(a_1,\ldots,a_m)(y)\,dy
\Big|
\lesssim
\frac{1}{t^n}
\int_{B(x,t)}
|{\mathcal T}_\sigma(a_1,\ldots,a_m)(y)|\,dy
\]
for each $t\in (0,\infty)$.
Similar to the previous section, it is enough to consider the following form:
\begin{equation}\label{eq.CalZygOPT-31}
{\mathcal T}_\sigma(f_1,\ldots,f_m)
=
\prod_{g=1}^{G}
T_{\sigma_{I_g}}(\{f_l\}_{l \in I_g}),
\end{equation}
where $\{I_g\}_{g=1}^{G}$ is a partition
of $\{1,\ldots,m\}$ with $1\in I_1$. 
By the H\"older inequality, we have
\begin{equation}
\label{eq.4B07}
\frac{1}{t^n}
\int_{B(x,t)}|{\mathcal T}_\sigma(a_1,\ldots,a_m)(y)|\,dy
\lesssim
\prod_{g=1}^G
\Big(
\frac{1}{t^n}
\int_{B(x,t)}
|T_{\sigma_{I_g}}(\{a_l\}_{l \in I_g})(y)|^{G}dy
\Big)^{\frac1{G}}.
\end{equation}
For each $1\le g\le G$, we need to examine
\[
\Big(
\frac{1}{t^n}
\int_{B(x,t)}
|T_{\sigma_{I_g}}(\{a_l\}_{l \in I_g})(y)|^{G}dy
\Big)^{\frac1{G}}.
\]
We consider two cases as in the proof of Lemma \ref{lm.4A00}.

\textbf{Case 1: $t\le \ell(Q_{1})$}.
We observe that  
\begin{align*}
\frac{1}{t^n}
\int_{B(x,t)}|{\mathcal T}_\sigma(a_1,\ldots,a_m)(y)|\,dy
&\le
\prod_{g: B(x,t)\cap Q_{l(g)}^{**}\neq\emptyset}
\Big(
\frac{1}{t^n}
\int_{B(x,t)}
|T_{\sigma_{I_g}}(\{a_l\}_{l \in I_g})(y)|^{G}dy
\Big)^{\frac1{G}}\\
&\quad \times 
\prod_{g: B(x,t)\cap Q_{l(g)}^{**}=\emptyset}
\Big(
\frac{1}{t^n}
\int_{B(x,t)}
|T_{\sigma_{I_g}}(\{a_l\}_{l \in I_g})(y)|^{G}dy
\Big)^{\frac1{G}}.
\end{align*}
When $B(x,t)\cap Q_{l(g)}^{**}\neq\emptyset$, 
we see that
$x\in 3Q_{l(g)}^{**}$. 
This shows 
\begin{align}\label{161112-4}
\lefteqn{
\prod_{g: B(x,t)\cap Q_{l(g)}^{**}\neq\emptyset}
\Big(
\frac{1}{t^n}
\int_{B(x,t)}
|T_{\sigma_{I_g}}(\{a_l\}_{l \in I_g})(y)|^{G}dy
\Big)^{\frac1{G}}
}\\
\nonumber
&\lesssim 
\prod_{g: B(x,t)\cap Q_{l(g)}^{**}\neq\emptyset}
\chi_{3Q_{l(g)}^{**}}(x)
M^{(G)} \circ T_{\sigma_{I_g}}(\{a_l\}_{l\in I_g})(x).
\end{align}
When $B(x,t)\cap Q_{l(g)}^{**}=\emptyset$, 
we may use \eqref{eq.3D4}
to have 
\begin{equation}\label{161112-3}
\prod_{g:B(x,t)\cap Q_{l(g)}^{**}=\emptyset}
\Big(
\frac{1}{t^n}
\int_{B(x,t)}
|T_{\sigma_{I_g}}(\{a_l\}_{l \in I_g})(y)|^{G}dy
\Big)^{\frac1{G}}
\lesssim 
\prod_{g: B(x,t)\cap Q_{l(g)}^{**}=\emptyset}
\prod_{l\in I_g}
M\chi_{Q_l}(x)^\frac{n+N+1}{mn}.
\end{equation}
These two estimates \eqref{161112-4} and \eqref{161112-3} yield the desired estimate in the Case 1.

\textbf{Case 2: $t> \ell(Q_{1})$}. We split
\begin{align*}
\lefteqn{
\frac{1}{t^n}
\int_{B(x,t)}|{\mathcal T}_\sigma(a_1,\ldots,a_m)(y)|\,dy
}\\
&\lesssim
\frac{1}{|Q_1^*|}
\int_{Q_1^*}
|{\mathcal T}_\sigma(a_1,\ldots,a_m)(y)|\,dy
+
\frac{1}{|Q_1^*|}
\int_{(Q_1^*)^c}
|{\mathcal T}_\sigma(a_1,\ldots,a_m)(y)|\,dy.
\end{align*}
For the first term, 
(\ref{161112-3}) yields
\begin{align*}
&
\frac{1}{|Q_1^*|}
\int_{Q_1^*}
|{\mathcal T}_\sigma(a_1,\ldots,a_m)(y)|\,dy\\
&\lesssim
\prod_{g=1}^G
\left(
\prod_{l\in I_g}
M\chi_{Q_l}(x)^\frac{n+N+1}{mn}
+
\chi_{Q_{l(g)}^{**}}(x)
M^{(G)} \circ T_{\sigma_{I_g}}(\{a_l\}_{l\in I_g})(x)
\right).
\end{align*}
For the second term, 
by a dyadic  decomposition of $(Q_1^*)^c$, 
\begin{align*}
&\frac{1}{|Q_1^*|}
\int_{(Q_1^*)^c}
|{\mathcal T}_\sigma(a_1,\ldots,a_m)(y)|\,dy\\
&=
\sum_{j=0}^{\infty}
\frac{1}{|Q_1^*|}
\int_{2^{j+1}Q_1^*\setminus 2^jQ_1^*}
|T_{\sigma_{I_1}}(\{a_l\}_{l\in I_1})(y)|
\prod_{g\geq 2}
|T_{\sigma_{I_g}}(\{a_l\}_{l\in I_g})(y)|dy
=
\sum_{j=0}^\infty I_j.
\end{align*}
Now, we fix any $j$ and evaluate each $I_j$. 
Letting 
$\Lambda_j=\{l=1,\ldots,m: 2^jQ_1^{**}\cap 2^jQ_l^{**}=\emptyset\}$ 
and using 
Lemma \ref{lm-161112-1},  for  
$y\in2^{j+1}Q_1^*\setminus2^jQ_1^*$ we obtain
\begin{align*}
|T_{\sigma_{I_1}}(\{a_l\}_{l\in I_1})(y)|
&\lesssim 
\left(
\frac{\ell(Q_1)}{|y-c_1|+\sum_{l\in I_1\cap \Lambda_j}|y-c_l|}
\right)^{n+N+1}\\
&\lesssim
2^{-j(n+N+1)}
\left(
\frac{2^j\ell(Q_1)}{2^j\ell(Q_1)+\sum_{l\in I_1\cap \Lambda_j}|c_1-c_l|}
\right)^{n+N+1} .
\end{align*} 
We estimate this term further. 
If $l\in I_1\cap \Lambda_j$, 
$|c_1-c_l|\sim |x-c_l|$ since $x\in Q_1^{**}$.
On the other hand, 
if $l\in I_1\setminus \Lambda_j$, 
$\chi_{2^jQ_l^{**}}(x)=\chi_{Q_1^{**}}(x)=1$ 
since $x\in Q_1^{**}$.
So, 
we have 
\begin{align*}
|T_{\sigma_{I_1}}(\{a_l\}_{l\in I_1})(y)|
&\lesssim 
2^{-j(n+N+1)}
\prod_{l\in I_1\cap \Lambda_j}
\left(
\frac{2^j\ell(Q_l)}{|x-c_l|}
\right)^\frac{n+N+1}{m}
\prod_{l\in I_1\setminus \Lambda_j}
\chi_{2^jQ_l^{**}}(x)\\
&\lesssim 
2^{-j(n+N+1)}
\prod_{l\in I_1\setminus \{1\}}
M\chi_{2^jQ_l^{**}}(x)^\frac{n+N+1}{mn}\\
&\lesssim
2^{-j(n+N+1)}
2^{j\frac{n+N+1}{m}(m_1-1)}
\prod_{l\in I_1}
M\chi_{Q_l}(x)^\frac{n+N+1}{mn}.
\end{align*}
This and H\"{o}lder's inequality imply that 
\begin{equation}\label{161112-1}
I_j
\lesssim 
2^{-j(N+1)}
2^{j\frac{n+N+1}{m}(m_1-1)}
\prod_{l\in I_1}
M\chi_{Q_l}(x)^\frac{n+N+1}{mn}
\prod_{g\geq2}
\left(
\frac{1}{|2^jQ_1^*|}
\int_{2^jQ_1^*}
|T_{\sigma_{I_g}}(\{a_l\}_{l\in I_g})(y)|^Gdy
\right)^\frac{1}{G}. 
\end{equation}
In the usual way, we claim that 
\begin{align}\label{161112-2}
&\prod_{g\geq2}
\left(
\frac{1}{|2^jQ_1^*|}
\int_{2^jQ_1^*}
|T_{\sigma_{I_g}}(\{a_l\}_{l\in I_g})(y)|^Gdy
\right)^\frac{1}{G}\\
&\lesssim 
2^{j\frac{n+N+1}{m}(m-m_1)}
\prod_{g\geq2}
\left(
M\chi_{Q_{l(g)}}(x)^\frac{(n+N+1)m_g}{nm}
M^{(G)} \circ T_{\sigma_{I_g}}(\{a_{l,k_l}\}_{l\in I_g})(x)
+
\prod_{l\in I_g}
M\chi_{Q_l}(x)^\frac{n+N+1}{mn}
\right).\nonumber
\end{align}
To see this, we again consider two possibilities of $g$ for each $j$; 
$2^jQ_1^{**}\cap 2^jQ_{l(g)}^{**}\neq \emptyset$ or 
not. 
In the first case, 
we notice $x\in Q_1^*\subset 2^jQ_{l(g)}^{**}$ 
and that we defined $m_g=\sharp I_g$,
and hence
\begin{align*}
\lefteqn{
\left(
\frac{1}{|2^jQ_1^*|}
\int_{2^jQ_1^*}
|T_{\sigma_{I_g}}(\{a_l\}_{l\in I_g})(y)|^Gdy
\right)^\frac{1}{G} 
}\\
&\lesssim 
\chi_{2^jQ_{l(g)}^{**}}(x)
M^{(G)} \circ T_{\sigma_{I_g}}(\{a_l\}_{l\in I_g})(x)\\
&\lesssim
2^{j\frac{m_g(n+N+1)}{m}}
M\chi_{Q_{l(g)}}(x)^\frac{m_g(n+N+1)}{mn}
M^{(G)} \circ T_{\sigma_{I_g}}(\{a_l\}_{l\in I_g})(x).
\end{align*}
In the second case; 
$2^jQ_1^{**}\cap 2^jQ_{l(g)}^{**}=\emptyset$, 
we use \eqref{eq.3D4} to see 
\begin{align*}
\lefteqn{
\left(
\frac{1}{|2^jQ_1^*|}
\int_{2^jQ_1^*}
|T_{\sigma_{I_g}}(\{a_l\}_{l\in I_g})(y)|^Gdy
\right)^\frac{1}{G} 
}\\
&\lesssim 
\prod_{l\in I_g}
M\chi_{2^jQ_l^{**}}(x)^\frac{n+N+1}{mn}
\lesssim
2^{j\frac{m_g(n+N+1)}{m}}
\prod_{l\in I_g}
M\chi_{Q_l}(x)^\frac{n+N+1}{mn}.
\end{align*}
These two estimates 
yields \eqref{161112-2}. 
Inserting \eqref{161112-2} to \eqref{161112-1}, 
we arrive at 
\[
I_j\lesssim 
2^{-j(\frac{n+N+1}{m}-n)}
\prod_{g=1}^G
\left(
M\chi_{Q_{l(g)}}(x)^\frac{m_g(n+N+1)}{mn}
M^{(G)} \circ T_{\sigma_{I_g}}(\{a_l\}_{l\in I_g})(x)
+
\prod_{l\in I_g}
M\chi_{Q_l}(x)^\frac{n+N+1}{mn}
\right). 
\]
Taking $N$ sufficiently large, 
we can sum the above estimate
up over $j \in {\mathbb N}$ and get desired estimate.
This completes the proof of Lemma \ref{lm.5A00}
\end{proof}

\begin{lemma}\label{lm.5A02}
Assume
$x \notin Q_1^{**}$ and $c_1 \notin B(x,100n^2t)$.
Then %if we choose $B^1_{l,0}=b_{l,0}$ satisfying $(\ref{eq:160726-1})$ suitably,
%then
we have 
\begin{align*}
&\frac{1}{t^n}
\int_{B(x,t)}|{\mathcal T}_\sigma(a_1,\ldots,a_m)(y)|\,dy\\
&\lesssim 
\prod_{g=1}^G
\left(
M\chi_{Q_{l(g)}}(x)^\frac{m_g(n+N+1)}{mn}
M^{(G)} \circ T_{\sigma_{I_g}}(\{a_l\}_{l\in I_g})(x)
+
\prod_{l\in I_g}
M\chi_{Q_l}(x)^\frac{n+N+1}{mn}
\right). 
\end{align*}
\end{lemma}
\begin{proof}
Fix $x\notin Q_1^{**}$ and $t>0$ such that $c_1\notin B(x,100n^2t).$ Let ${\mathcal T}_\sigma$ be the operator of type \eqref{eq.CalZygOPT-3}.
We may consider the  reduced form \eqref{eq.CalZygOPT-31} of $\mathcal{T}_{\sigma}$ and start from \eqref{eq.4B07}.
We define 
\[
J=
\{
g=2,\ldots, m: 
x\notin Q_{l(g)}^{**}
\},
\quad
J_0=
\{
g\in J: 
B(x,2t)\cap Q_{l(g)}^*=\emptyset
\},
\quad
J_1=J\setminus J_0
\]
and split the product as follows:
\begin{align*}
\frac{1}{t^n}
\int_{B(x,t)}|{\mathcal T}_\sigma(a_1,\ldots,a_m)(y)|\,dy
&\lesssim
\left\|
T_{\sigma_{I_1}}(\{a_l\}_{l\in I_1})
\chi_{B(x,t)}
\right\|_{L^\infty}
\prod_{l\in J_0}
\left\|
T_{\sigma_{I_g}}(\{a_l\}_{l\in I_g})
\chi_{B(x,t)}
\right\|_{L^\infty}\\
&\quad \times 
\prod_{g\in J_1}
\left(
\frac{1}{|B(x,t)|}
\int_{B(x,t)}
|T_{\sigma_{I_g}}(\{a_l\}_{l\in I_g})(y)|^Gdy
\right)^\frac{1}{G}\\
&\quad \times 
\prod_{g\in \{2,\ldots,G\} \setminus J}
\left(
\frac{1}{|B(x,t)|}
\int_{B(x,t)}
|T_{\sigma_{I_g}}(\{a_l\}_{l\in I_g})(y)|^Gdy
\right)^\frac{1}{G}\\
&=
{\rm I}
\times 
{\rm II}
\times 
{\rm III}
\times 
{\rm IV}.
\end{align*}
To estimate I, we further define the partition of $I_1$: 
\begin{align*}
I_1^0 
&= \{l\in I_1 : x\notin Q_l^{**}, B(x,2t)\cap Q_l^*=\emptyset\},
\quad
I_1^1 
= \{l\in I_1 : x\notin Q_l^{**}, B(x,2t)\cap Q_l^*\ne\emptyset\},\\
I_1^2 
&= I_1\setminus(I_1^0\cup I_1^1).
\end{align*}
Since $x\notin Q_1^{**}$ and $c_1\notin B(x,100n^2t)$, we can see that $1\in I_1^0$. 
From Lemma \ref{lm.3A1}, we deduce
\begin{align*}
&|T_{\sigma_{I_1}}(\{a_l\}_{l \in I_1})(y)|\lesssim
\frac{\ell(Q_1)^{n+N+1}}{(\sum_{l\in I_1^0}|y-c_l|)^{n+N+1}}
\lesssim
\frac{\ell(Q_1)^{n+N+1}}{(\sum_{l\in I_1^0}|x-c_l|)^{n+N+1}}
\\
&\lesssim
\Big(\frac{\ell(Q_1)}{|x-c_1|+\ell(Q_1)}\Big)^{(m-m_1)\frac{n+N+1}{m}}
\prod_{l\in I_1^0}
\Big(\frac{\ell(Q_l)}{|x-c_l|+\ell(Q_l)}\Big)^{\frac{n+N+1}{m}}
\prod_{l\in I_1^1}
\Big(\frac{\ell(Q_1)}{|x-c_1|+\ell(Q_1)}\Big)^{\frac{n+N+1}{m}}
\end{align*}
for all $y\in B(x,t)$, where $m_1=|I_1|$ is the cardinality of the set $I_1$.
As in the proof of Lemma  \ref{lm.4F01}  for the product type, if $x\notin Q_l^{**}$ and $B(x,2t)\cap Q_l^{*}\ne\emptyset$ then $|x-c_l|\lesssim t\lesssim |x-c_1|$. This observation implies
\[
\frac{\ell(Q_1)}{|x-c_1|+\ell(Q_1)}
\lesssim
\frac{\ell(Q_l)}{|x-c_l|+\ell(Q_l)}
\]
for all $l\in I_1^1$. Therefore, we can estimate
\[
|T_{\sigma_{I_1}}(\{a_l\}_{l \in I_1})(y)|
\lesssim
\Big(\frac{\ell(Q_1)}{|x-c_1|+\ell(Q_1)}\Big)^{(m-m_1)\frac{n+N+1}{m}}
\prod_{l\in I_1^0\cup I_1^1}M\chi_{Q_l}(x)^{\frac{n+N+1}{mn}}
\]
for all $y\in B(x,t)$. Obviously, $1\lesssim M\chi_{Q_l}(x)$ for all $l\in I_1^2$, and hence we have
\begin{equation}
\label{eq.5B08}
|T_{\sigma_{I_1}}(\{a_l\}_{l \in I_1})(y)|
\lesssim
\Big(\frac{\ell(Q_1)}{|x-c_1|+\ell(Q_1)}\Big)^{(m-m_1)\frac{n+N+1}{m}}
\prod_{l\in I_1}M\chi_{Q_l}(x)^{\frac{n+N+1}{mn}}
\end{equation}
for all $y\in B(x,t)$ which gives the estimate for I. 
For the third term III, we simply have 
\[
{\rm III}
\leq 
\prod_{g\in J_1}
M^{(G)} \circ T_{\sigma_{I_g}}(\{a_l\}_{l\in I_g})(x).
\]
So, we obtain 
\begin{align*}
{\rm I}\times {\rm III}
&\lesssim 
\prod_{l\in I_1}M\chi_{Q_l}(x)^\frac{n+N+1}{mn}
\prod_{g\in J_1}
\left(
\frac{\ell(Q_1)}
{
\ell(Q_1)+
|x-c_1|
}
\right)^\frac{m_g(n+N+1)}{m}
M^{(G)} \circ T_{\sigma_{I_g}}(\{a_l\}_{l\in I_g})(x)\\
&\lesssim 
\prod_{l\in I_1}M\chi_{Q_l}(x)^\frac{n+N+1}{mn}
\prod_{g\in J_1}
M\chi_{Q_{l(g)}}(x)^\frac{m_g(n+N+1)}{mn}
M^{(G)} \circ T_{\sigma_{I_g}}(\{a_l\}_{l\in I_g})(x), 
\end{align*}
since $g\in J_1$ implies 
$|x-c_{l(g)}|\lesssim|x-c_1|$. 
For the second term II, 
we use Lemma \ref{161112-1}
and an argument as for estimate for I to get 
\[
{\rm II}=
\prod_{l\in J_0}
\left\|
T_{\sigma_{I_g}}(\{a_l\}_{l\in I_g})
\chi_{B(x,t)}
\right\|_{L^\infty}
\lesssim 
\prod_{g\in J_0}
\prod_{l\in I_g}M\chi_{Q_l}(x)^\frac{n+N+1}{mn}.
\]
For the last term IV, we recall $g\notin J$ means $x\in Q_{l(g)}^{**}$ 
and hence, 
\[
{\rm IV}
\lesssim
\prod_{g\notin J}
M\chi_{Q_{l(g)}}(x)^\frac{m_g(n+N+1)}{mn}
M^{(G)} \circ T_{\sigma_{I_g}}(\{a_l\}_{l\in I_g})(x).
\]
Combining the estimates for I, II, III and IV, 
we complete the proof of Lemma \ref{lm.5A02}.
\end{proof}

\begin{lemma}
\label{lm.5C05}
Assume
$x \notin Q_1^{**}$ and $c_1 \in B(x,100n^2t)$.
Then we have 
\begin{align*}
\lefteqn{
\frac{\ell(Q_1)^{s+1}}{t^{n+s+1}}
\int_{Q_1^{*}}|{\mathcal T}_\sigma(a_1,\ldots,a_m)(y)|\,dy
}\\
&
\lesssim M\chi_{Q_1}(x)^{\frac{n+s+1}{n}}
\prod_{g=1}^G
\inf_{z\in Q_1^*} \!\!
\left[
M\chi_{Q_{l(g)}}(z)^\frac{m_g(n+N+1)}{mn}
M^{(G)} \circ T_{\sigma_{I_g}}(\{a_l\}_{l\in I_g})(z)
+
\prod_{l\in I_g}
M\chi_{Q_l}(z)^\frac{n+N+1}{mn}
\right] \!. 
\end{align*}
\end{lemma}

\begin{lemma}
\label{lm.5C06}
Assume
$x \notin Q_1^{**}$ and $c_1 \in B(x,100n^2t)$.
Then we have 
\begin{align*}
\lefteqn{
\frac{1}{t^{n+s+1}}
\int_{(Q_1^{*})^c}|y-c_1|^{s+1}|{\mathcal T}_\sigma(a_1,\ldots,a_m)(y)|\,dy
}\\
&\quad
\lesssim M\chi_{Q_1}(x)^{\frac{n+s+1}{n}}
\prod_{g=1}^G
\inf_{z\in Q_1^*} \!\!
\left[
M\chi_{Q_{l(g)}}(z)^\frac{m_g(n+N+1)}{mn}
M^{(G)} \circ T_{\sigma_{I_g}}(\{a_l\}_{l\in I_g})(z)
+
\prod_{l\in I_g}
M\chi_{Q_l}(z)^\frac{n+N+1}{mn}
\right]\! . 
\end{align*}
\end{lemma}

The proof of Lemmas \ref{lm.5C05} and \ref{lm.5C06} are very similar to those of Lemma \ref{lm.5A00}, so we omit the details here.

\subsection{The proof of Proposition \ref{LM.Key-31} for the mixed type}
Employing the above lemmas, we   complete the proof of \eqref{eq.PWEST}. 
For each $\vec{k}=(k_1,\ldots, k_m)$, 
recall the smallest-length cube $R_{\vec{k}}$ among $Q_{1,k_1},\ldots, Q_{m,k_m}$ 
and write 
$Q_{l(g),\vec{k}(g)}$ for the  
 cube of smallest-length among 
$\{Q_{l,k_l}\}_{l\in I_g}$. 
Combining Lemmas \ref{lm.5A00}-\ref{lm.5C06}, 
we have the following pointwise estimate 
\begin{align*}
&M_\phi \circ {\mathcal T}_\sigma(a_{1,k_1},\ldots,a_{m,k_m})(x)
\lesssim
\prod_{g=1}^G
b_{g,\vec{k}(g)}(x)
+ 
M\chi_{R_{\vec{k}}^*}(x)^{\frac{n+s+1}{n}}
\prod_{g=1}^G
\inf_{z\in R_{\vec{k}}^*}
b_{g,\vec{k}(g)}(z),\\
&b_{g,\vec{k}(g)}(x)
=
%\left(
M\chi_{Q_{l(g),\vec{k}(g)}^*}(x)^\frac{m_g(n+N+1)}{mn}
M^{(G)} \circ T_{\sigma_{I_g}}(\{a_{l,k_l}\}_{l\in I_g})(x)
+
\prod_{l\in I_g}
M\chi_{Q_{l,k_l}}(x)^\frac{n+N+1}{mn}
%\right) 
\end{align*}
for all $x\in \R^n$.
As in the proof for the product type, we let 
\begin{equation*}
A= \Big\|\sum_{k_1,\ldots,k_m=1}^\infty
\Big(\prod_{l=1}^m \lambda_{l,k_l}\Big)
M_\phi \circ {\mathcal T}_\sigma(a_{1,k_1},\ldots,a_{m,k_m})\Big\|_{L^p}.
\end{equation*}
In view of $(n+s+1)p/n>1$, using Lemma \ref{lm.2B00} and H\"{o}lder's inequality, 
we see 
\begin{align*}
A
&\lesssim 
\prod_{g=1}^G
\bigg\|
\sum_{k_l\ge 1:l\in I_g}
\bigg(\prod_{l\in I_g}\lambda_{l,k_l}\bigg)
\bigg(
\Big(M\chi_{Q_{l(g),\vec{k}(g)}^*}\Big)^\frac{m_g(n+N+1)}{mn}
 M^{(G)} \circ T_{\sigma_{I_g}}(\{a_{l,k_l}\}_{l\in I_g})
 \\
& \qquad \qquad \qquad \qquad \qquad \qquad\qquad \qquad \qquad\qquad \qquad \qquad+
\prod_{l\in I_g}
(M\chi_{Q_{l,k_l}})^\frac{n+N+1}{mn}
\bigg)
\bigg\|_{L^{q_g}}\\ 
&\lesssim 
\prod_{g=1}^G
\Big(
A_{g,1}+
A_{g,2}
\Big), 
\end{align*}
where $q_g\in (0,\infty)$ is defined by 
$
1/q_g
=
\sum_{l\in I_g}
1/p_l
$  and  
\begin{align*}
A_{g,1}&=
\left\|
\sum_{k_l\ge 1 : l\in I_g}
\left(\prod_{l\in I_g}\lambda_{l,k_l}\right)
\Big(M\chi_{Q_{l(g),\vec{k}(g)}^*}\Big)^\frac{m_g(n+N+1)}{mn}
 M^{(G)}\circ T_{\sigma_{I_g}}(\{a_{l,k_l}\}_{l\in I_g})
\right\|_{L^{q_g}},\\ 
A_{g,2}&=
\left\|
\sum_{k_l\ge 1 : l\in I_g}
\prod_{l\in I_g}\lambda_{l,k_l}
(M\chi_{Q_{l,k_l}})^\frac{n+N+1}{mn}
\right\|_{L^{q_g}}.  
\end{align*}
For $A_{g,2}$, we have only to employ Lemma \ref{lm.2B00} to get 
the desired estimate. 
For $A_{g,1}$, take large $r$ and  
employ Lemma \ref{lm-161031-1} to obtain  
\[
A_{g,1}
\lesssim 
\left\|
\sum_{k_l\ge 1 : l\in I_g}
\left(\prod_{l\in I_g}\lambda_{l,k_l}\right)
\chi_{Q_{l(g),\vec{k}(g)}^*}
M^{(r)}\circ M^{(G)}
[T_{\sigma_{I_g}}(\{a_{l,k_l}\}_{l\in I_g})]
\right\|_{L^{q_g}}.
\]
Then it follows from Lemma \ref{lm.2A00} 
and \eqref{eq.3A3} 
that 
\begin{align*}
A_{g,1}
&\lesssim 
\left\|
\sum_{k_l\ge 1 : l\in I_g}
\left(\prod_{l\in I_g}\lambda_{l,k_l}\right)
\frac{\chi_{Q_{l(g),\vec{k}(g)}^*}}{|Q_{l(g),\vec{k}(g)}|^{1/q}}
\left\|
\chi_{Q_{l(g),\vec{k}(g)}^*}
M^{(r)}\circ M^{(G)}
[T_{\sigma_{I_g}}(\{a_{l,k_l}\}_{l\in I_g})]
\right\|_{L^q}
\right\|_{L^{q_g}}\\ 
&\lesssim 
\left\|
\sum_{k_l\ge 1 : l\in I_g}
\left(\prod_{l\in I_g}\lambda_{l,k_l}\right)
\chi_{Q_{l(g),\vec{k}(g)}^*}
\inf_{z\in Q_{l(g),\vec{k}(g)}^*}
\prod_{l\in I_g}
M\chi_{Q_l}(z)^\frac{n+N+1}{mn}
\right\|_{L^{q_g}}\\ 
&\leq
\prod_{l\in I_g}
\left\|
\sum_{k_l=1}^\infty
\lambda_{l,k_l}
(M\chi_{Q_l})^\frac{n+N+1}{mn}
\right\|_{L^{p_l}}
\lesssim 
\prod_{l\in I_g}
\left\|
\sum_{k_l=1}^\infty
\lambda_{l,k_l}
\chi_{Q_l}
\right\|_{L^{p_l}}, 
\end{align*}
which completes the proof of Lemma \ref{LM.Key-31} for for operators of mixed type.

%=============================================================%
\section{Examples}
%=============================================================%

We provide some examples of operators of the kinds discussed in this paper: all of the following 
are symbols of trilinear operators acting on functions on the real line, thus they are functions on 
$\mathbb R^3= \mathbb R\times \mathbb R\times \mathbb R$. 

The symbol 
$$
\sigma_1(\xi_1,\xi_2,\xi_3) = \frac{ (\xi_1+\xi_2+\xi_3)^2 }{ \xi_1^2+\xi_2^2+\xi_3^2} 
$$
is associated with an operator of type~\eqref{eq.CalZygOPT}.

The symbol 
\begin{align*}
\sigma_2 (\xi_1,\xi_2,\xi_3)
&= 
\frac{\xi_1^3}{(1+\xi_1^2)^{\frac 32}} \frac{1}{(1+\xi_2^2+\xi_3^2)^{\frac 32}} +
\frac{1}{(1+\xi_1^2)^{\frac 32}} \frac{\xi_2^3}{(1+\xi_2^2+\xi_3^2)^{\frac 32}} \\
&\quad
+\frac{1}{(1+\xi_1^2)^{\frac 32}} \frac{\xi_3^3}{(1+\xi_2^2+\xi_3^2)^{\frac 32}} 
- \frac{3\xi_1} {(1+\xi_1^2)^{\frac 32}} \frac{\xi_2\xi_3}{(1+\xi_2^2+\xi_3^2)^{\frac 32}} \\
&=
\frac{(\xi_1+\xi_2+\xi_3)(\xi_1^2+\xi_2^2+\xi_3^2-\xi_1\xi_2-\xi_2\xi_3-\xi_3\xi_1)}
{(1+\xi_1^2)^{\frac 32}(1+\xi_2^2+\xi_3^2)^{\frac 32}}
\end{align*}
provides an example of an operator of type ~\eqref{eq.CalZygOPT-3}. 
Note that each term is given as a product of a multiplier of 
$\xi_1$ times a multiplier of $(\xi_2,\xi_3)$.

The symbol 
\begin{eqnarray*}
\sigma_3 (\xi_1,\xi_2,\xi_3) &= & \frac{ \xi_1^4}{(1+\xi_1^2)^2} 
\frac{\xi_2^2}{(1+\xi_2^2)^2} \frac{\xi_3} { (1+\xi_3^2)^2} - 
 \frac{ \xi_1^4}{(1+\xi_1^2)^2} \frac{\xi_2}{(1+\xi_2^2)^2} 
\frac{ \xi_3^2} { (1+\xi_3^2)^2} \\
&& \quad- \frac{\xi_1^2}{(1+\xi_1^2)^2}\frac{ \xi_2^4}{(1+\xi_2^2)^2}
 \frac{\xi_3}{ (1+\xi_3^2)^2} +
 \frac{ \xi_1}{(1+\xi_1^2)^2} \frac{ \xi_2^4}{(1+\xi_2^2)^2} 
\frac{\xi_3^2} { (1+\xi_3^2)^2} \\
&& \quad+
\frac{\xi_1^2}{(1+\xi_1^2)^2}\frac{\xi_2} {(1+\xi_2^2)^2} 
\frac{ \xi_3^4}{(1+\xi_3 ^2)^2} 
 - \frac{\xi_1}{(1+\xi_1^2)^2}\frac{\xi_2^2} {(1+\xi_1^2)^2} \frac{ \xi_3^4}{(1+\xi_3 ^2)^2}\\
&=&-
\frac{\xi_1\xi_2\xi_3(\xi_1-\xi_2)(\xi_2-\xi_3)(\xi_3-\xi_1)
(\xi_1+\xi_2+\xi_3)}{(1+\xi_1^2)^2(1+\xi_2^2)^2(1+\xi_3^2)^2}
\end{eqnarray*}
yields an example of an operator of type ~\eqref{eq.CalZygOPT-2}. 
The next example:
\begin{eqnarray*}
\sigma_4 (\xi_1,\xi_2,\xi_3) &= &
\frac{\xi_1\xi_2}{\xi_1^2+\xi_2^2+(\xi_1+\xi_2)^2}
\cdot 1
-
\frac{\xi_1\xi_2}{\xi_1^2+\xi_2^2+\xi_3^2}
\end{eqnarray*}
shows that the integer $G(\rho)$ varies
according to $\rho$.
Notice that all four examples satisfy
$$
\sigma_1(\xi_1,\xi_2,\xi_3)=\sigma_2(\xi_1,\xi_2,\xi_3)
=\sigma_3(\xi_1,\xi_2,\xi_3)=\sigma_4(\xi_1,\xi_2,\xi_3)=0
$$
when $\xi_1+\xi_2+\xi_3=0$. 
This yields condition \eqref{eq.TmCan} when $s=0$; see \cite{Our-next-paper}.
For the case of $s\in \mathbb Z^+$,
we consider
$\sigma_1{}^{s+1},\sigma_2{}^{s+1},\sigma_3{}^{s+1}$,
for example.

\end{document}